\documentclass{amsart}
\usepackage{amsmath,amsthm,amscd,amssymb,latexsym,upref,hyperref,enumerate,xfrac}

\newcommand*{\mailto}[1]{\href{mailto:#1}{\nolinkurl{#1}}}

\newtheorem{theorem}{Theorem}[section]
\newtheorem{corollary}[theorem]{Corollary}

\newtheorem{lemma}[theorem]{Lemma}
\newtheorem{proposition}[theorem]{Proposition}

\theoremstyle{definition}
\newtheorem{definition}[theorem]{Definition}
\newtheorem{remark}[theorem]{Remark}
\newtheorem{assumption}[theorem]{Assumption}
\newtheorem*{notation}{Notation}

\newcommand{\ti}{\tilde}
\newcommand{\A}{\mathcal{A}}
\newcommand{\cE}{\mathcal{E}}
\newcommand{\E}{\mathrm{e}}
\newcommand{\I}{\mathrm{i}}

\newcommand{\re}{\mathrm{Re}}
\newcommand{\spr}[2]{\langle #1 , #2 \rangle}
\newcommand{\norm}[1]{\| #1 \|}
\newcommand{\hnorm}[1]{\| #1 \|_{\mathcal{H}}}

\newcommand{\sig}{\sigma}
\newcommand{\lam}{\lambda}
\newcommand{\om}{\omega}
\newcommand{\Om}{\Omega}

\newcommand{\R}{{\mathbb R}}
\newcommand{\N}{{\mathbb N}}

\newcommand{\C}{{\mathbb C}}
\newcommand{\h}{\mathcal{H}}
\newcommand{\D}{\mathcal{D}}
\newcommand{\DA}{\mathcal{D}(\mathcal{A})}
\newcommand{\lii}{{L^2(\Omega)}}
\newcommand{\hi}{{H^1(\Omega)}}
\newcommand{\hio}{{H_0^1(\Omega)}}
\newcommand{\hii}{{H^2(\Omega)}}
\newcommand{\hiii}{{H^3(\Omega)}}
\newcommand{\liilii}{L^2(0,T;L^2(\Omega))}
\newcommand{\liihii}{L^2(0,T;H^2(\Omega))}
\newcommand{\liihi}{L^2(0,T;H^1(\Omega))}
\newcommand{\clii}{C(0,T;L^2(\Om))}

\newcommand{\chio}{C(0,T;H_0^1(\Om))}

\newcommand{\cihii}{C^1(0,T;H^2(\Om))}

\newcommand{\ciilii}{C^2(0,T;L^2(\Om))}
\newcommand{\hilii}{H^1(0,T;L^2(\Om))}
\newcommand{\hiihi}{H^2(0,T;H^1(\Om))}

\newcommand{\hiiilii}{H^3(0,T;L^2(\Om))}

\numberwithin{equation}{section}

\begin{document}
\title[Well-posedness and asymptotic behavior]{Well-posedness and asymptotic behavior of solutions for the 
Blackstock--Crighton--Westervelt equation}

\author[R.\ Brunnhuber]{Rainer Brunnhuber}
\address{Insitut f\"ur Mathematik \\ Universit\"at Klagenfurt\\
Universit\"atsstra{\ss}e 65-57\\ 9020 Klagenfurt am W\"orthersee\\ Austria}
\email{\mailto{rainer.brunnhuber@aau.at}}
\urladdr{\url{http://www.aau.at/~rabrunnh/}}

\author[B.\ Kaltenbacher]{Barbara Kaltenbacher}
\address{Insitut f\"ur Mathematik\\ Universit\"at Klagenfurt\\
Universit\"atsstra{\ss}e 65-57\\ 9020 Klagenfurt am W\"orthersee\\ Austria}
\email{\mailto{barbara.kaltenbacher@aau.at}}
\urladdr{\url{http://www.aau.at/~bkaltenb/}}


\begin{abstract}We consider a nonlinear fourth order in space 
partial differential equation arising in the context of the modeling of nonlinear acoustic wave propagation in thermally 
relaxing viscous fluids. 

We use the theory of operator semigroups in order to investigate the linearization of the underlying model and
see that the underlying semigroup is analytic. This leads to exponential decay results for the linear homogeneous equation.

Moreover, we prove local in time well-posedness of the model under the assumption that initial data are sufficiently small 
by employing a fixed point argument.
Global in time well-posedness is obtained by performing energy estimates and using the classical barrier method, again for sufficiently small initial data.

Additionally, we provide results concerning exponential decay of solutions of the nonlinear equation.
\end{abstract}

\thanks{Research supported by the Austrian Science Fund (FWF) under Grant P24970}

\maketitle


\section{Introduction}
The mathematical modeling of nonlinear acoustic wave propagation is a highly active field of research
(see, e.g., \cite{Jordan:2004}, \cite{KL09}, \cite{KL10}, \cite{KL11}, \cite{KLM11}, \cite{KLP12}, \cite{Rozanova:2007} and \cite{Tjotta:2001}) recently
driven by the broad range of possible applications such as the medical and industrial use of high
intensity focused ultrasound (HIFU) in lithotripsy, thermotherapy, ultrasound cleaning and sonochemistry.

The classical model equations in nonlinear acoustics are the Kuznetsov equation,
the Westervelt equation and the Kokhlov-Zabolotskaya-Kuznetsov equation which are all of second order in
time and characterized by the presence of a viscoelastic damping.

The most general of these popular models is Kuznetsov's equation
\begin{equation}\label{kuznetsov}
\psi_{tt}-c^2 \Delta \psi - b\Delta \psi_t = \left(\frac{1}{c^2} \frac{B}{2A} \left(\psi_t\right)^2+|\nabla \psi|^2\right)_t
\end{equation}
for the acoustic velocity potential $\psi$, where $c>0$ is the speed of sound, $b\geq 0$ is the diffusivity of sound
and $B/A$ is the parameter of nonlinearity. Neglecting local nonlinear effects (in the sense that the expression
$c^2\vert\nabla \psi\vert^2-(\psi_t)^2$ is sufficiently small) one arrives at the Westervelt equation
\begin{equation}\label{westervelt}
\psi_{tt}-c^2\Delta \psi - b\Delta \psi_t = \left(\frac{1}{c^2}\left(1+\frac{B}{2A}\right)(\psi_t)^2\right)_t.
\end{equation}

The Kuznetsov equation in its turn can in some sense be regarded as a simplification (for a small ratio $\nu\mbox{Pr}^{-1}$ between the kinematic
viscosity $\nu$ and the Prandtl number $\mbox{Pr}$) of the following fourth order in space equation
\begin{equation} \label{fourthordereq}
\left(\frac{\nu}{\mbox{Pr}}\Delta - \partial_t \right)
\left(\psi_{tt}-c^2\Delta\psi - b\Delta\psi_t \right) =
\left(\frac{1}{c^2} \frac{B}{2A}(\psi_t)^2 + \left|\nabla\psi \right|^2 \right)_{tt}.
\end{equation}
This equation results from the following two general models from the original paper \cite{Crighton:1979} (see equations (11), (13) there)
\begin{equation}\label{Crighton11}
-c^2\frac{\nu}{\mbox{Pr}}\Delta^2\psi+\left(\frac{\nu}{\mbox{Pr}}+b\right)\Delta\psi_{tt}
+\left(c^2\Delta\psi_t-\psi_{ttt}\right)
=\left(\left|\nabla\psi\right|^2_t+\frac{B}{2A}\psi_t\Delta\psi\right)_t
\end{equation}
and
\begin{equation} \label{Crighton13}
\left(\frac{\nu}{\mbox{Pr}}\Delta - \partial_t \right)
\left(c^2\Delta\psi -\psi_{tt} \right)
=-\left(\left|\nabla\psi\right|^2_t+\frac{B}{2A}\psi_t\Delta\psi\right)_t
\end{equation}
We replace of $\Delta \psi$ in the last term of \eqref{Crighton11}, \eqref{Crighton13} by $\frac{1}{c^2}\psi_{tt}$,
which can be justified by the main part of the differential operator (that corresponds to the wave equation $\psi_{tt}-c^2 \Delta \psi =0$).
Moreover, we consider potential diffusivity as appearing in \eqref{Crighton11}.
Therewith, equation \eqref{Crighton13} becomes \eqref{fourthordereq}.
We call \eqref{fourthordereq} Blackstock--Crighton--Kuznetsov equation.

Neglecting local nonlinear effects which are taken into account
by the gradient on the right-hand side (as it is done when reducing the Kuznetsov to the
Westervelt equation) one arrives at
\begin{equation}\label{fourthordereqnograd}
\left(\frac{\nu}{\mbox{Pr}}\Delta - \partial_t \right)
\left(\psi_{tt}-c^2\Delta\psi - b\Delta\psi_t \right) =
\left(\frac{1}{c^2}\left(1+\frac{B}{2A}\right)(\psi_t)^2\right)_{tt}
\end{equation}
which we call Blackstock--Crighton--Westervelt equation.

We abbreviate $a=\nu \mbox{Pr}^{-1}$ and $\sigma=\frac{1}{c^2} \left(1+\frac{B}{2A}\right)$.
Our object of investigation is the initial boundary value problem
\begin{equation}\label{abbreq}
\begin{cases}
(a\Delta - \partial_t)\left(\psi_{tt}-c^2\Delta\psi - b\Delta\psi_t \right)=\sigma(\psi_t^2)_{tt}  & \text{ in } \Omega \times (0,T]
\\
(\psi,\psi_t,\psi_{tt})=(\psi_0, \psi_1, \psi_2) & \text{ on } \Omega \times \{t=0\}
\\
\psi =0  & \text{ on } \partial\Omega \times [0,T]
.\\
\end{cases}
\end{equation}
on an open and bounded subset $\Om\subset \R^d, d\in\{1,2,3\}$ with smooth boundary $\partial \Omega$, where $\psi_0, \psi_1, \psi_2: \Om \rightarrow \R$ are given and
$\psi:\overline{\Om}\times [0,T] \rightarrow \R$ is the unknown, $\psi=\psi (x,t)$.

The restriction on the dimension of $\Om$ is imposed in order to be able to use various embedding theorems.
Increasing the space dimension to $d\geq 4$ would not be of relevance in applications anyway.

Throughout this paper we will assume $b>0$ since this essential for several results we intend we prove (cf. Remarks \ref{rem:b} and \ref{rem:spectrum}).

We will prove local and global in time well-posedness and show that solutions decay exponentially with respect to certain norms.

In Section 2 we introduce the notation which will be used throughout this paper. We mention function spaces we
will make use of, e.g., $L^p$-spaces, Sobolev and Besov spaces and recall (respectively, refer to) embedding theorems needed for our proofs.

In Section 3 we consider the linearization of \eqref{fourthordereqnograd} in a general abstract form. We apply the
theory of operator semigroups to \eqref{fourthordereqnograd} and show that the underlying semigroup is analytic
on two different phase spaces which leads, together with certain spectral properties of the generator, to two
exponentially decaying energy functionals. Moreover, we provide existence and uniqueness results for the
solutions of the linear model.

In Section 4 we perform energy estimates as a preparation for the proof of global in time well-posedness.

Section 5 is devoted to the fully nonlinear equation \eqref{fourthordereqnograd}. We prove local in time well-posedness
by employing a fixed point argument. The space which we use in this fixed point argument is obtained by
combining regularity results for the linearized Westervelt equation and the heat equation. We achieve local existence and uniqueness
of solutions provided the given initial data are sufficiently small.
Global in time well-posedness is obtained by using the energy estimates from Section 4 together with classical barrier's method
which finally leads to an exponential decay result for the higher order energy functional introduced in Section 4.

Appendix A provides an overview of facts from the theory of operator semigroups (cf. \cite{EngelNagel:1999}, \cite{Pazy:1983}) which are
used in Section 3.


\section{Notation and preliminaries}
Suppose $\Om \subset \R^d$, $d\in\{1,2,3\}$ is an open and bounded domain with smooth boundary $\partial \Om$.

We denote by $L^p(\Om)$ the space of (classes of) Lebesgue integrable functions $\Om \rightarrow \R$ with exponent $p\in[1,\infty]$.
The norm of a function $u\in L^p(\Om)$ will be denoted by $\|u\|_{L^p(\Om)}$.
In the special case $p=2$ we simply write $\|u\|:=\|u\|_{L^2(\Om)}$ for the norm of a function $u\in L^2(\Om)$ and $\spr{u}{v}:=\spr{u}{v}_{L^2(\Om)}$ for
the inner product of $u,v \in L^2(\Om)$.

More generally, we will always write $\|u\|_X$ for the $X$-norm of a function $u\in X$ and $C_{X\hookrightarrow Y}$ for the norm of the
embedding $X\hookrightarrow Y$.

The space $C^k(0,T;X)$ consists of all $k$-times continuously differentiable functions $u: [0,T] \rightarrow X$ where $k\in \N_0$.

By $\h$ we denote a (separable) Hilbert space equipped with the inner product $\spr{.}{..}_\h$ and the induced norm $\hnorm{.}$.
In applications, we always think of $\h=L^2(\Om)$.

We write $W^{m,p}(\Om)$ for the (fractional) Sobolev space of order $m\in\R_+$ and exponent $p\in [1,\infty)$. As usual,
$H^m(\Om):=W^{m,2}(\Om)$. The space $H_0^1(\Om)$ contains all functions in $\hi$ with zero trace. For additional information on
Sobolev spaces, in particular embedding theorems, we refer to \cite{AdamsFournier:2003}
and \cite{DiNezzaPalatucciValdinoci:2011}.

Moreover, by $\A: \DA \rightarrow \h$ we denote a self-adjoint, strictly positive and closed operator whose domain of definition
$\DA$ is dense in $\h$ and always keep in mind that in applications we think of $\A = -\Delta$ being the negative Dirichlet Laplacian defined on
$\DA= H^2(\Om) \cap H_0^1(\Om)$.
\begin{remark}
As $\A$ was assumed to be strictly positive, we conclude that the spectrum of $\A$, $\sig(\A) \subset \R_+$, is bounded
from below and may therefore set $\underline{\mu}:=\min\sig(\A) >0$. Moreover, on the strength of positivity,
fractional powers $\A^\Theta$, $\Theta \geq 0$ of $\A$ are well-defined and $\A^\Theta$ is again a
strictly positive self-adjoint operator with domain of definition $\D(\A^\Theta)$.
Note that $\D(\A^{\Theta_1}) \subset \D(\A^{\Theta_2})$ for $\Theta_1 > \Theta_2$.
Furthermore, as $\DA$ was assumed to be dense in $\h$, we also have that
$\D(\A^\Theta)$ is dense in $\h$ for $0<\Theta<1$.
\end{remark}
\begin{assumption}\label{Aembeddings}
We assume the embeddings
\begin{align}
\D(\A^{\nu}) &\hookrightarrow \h, \qquad&& \mbox{with } \hnorm{v} \leq C_{\D(\A^\nu) \hookrightarrow \h} \hnorm{\A^\nu v}, \nu=\tfrac12, 1, \tfrac32,\label{Poincare}
\end{align}\begin{align}
\DA &\hookrightarrow L^\infty(\Om), && \mbox{with } \norm{v}_{L^\infty(\Om)} \leq C_{\DA \hookrightarrow L^\infty} \hnorm{\A v}. \label{linfDA}
\end{align}
Note that, in applications, \eqref{Poincare} corresponds to Poincar\'e's inequality, i.e. to the (repeatedly used) embedding
$H_0^1(\Om) \hookrightarrow L^2(\Om)$, and \eqref{linfDA} corresponds to the embedding $H^2(\Om) \cap\hio\hookrightarrow L^\infty(\Om)$.
\end{assumption}

Next, we recall the concept of Besov spaces and list embeddings we will make use of later on.
\begin{definition}[\cite{AdamsFournier:2003}]\label{def:Besov}
(i) Suppose $X_i$, $i=1,2$ are Banach spaces equipped with the norms $\|.\|_i$, $i=1,2$.
We say that a Banach space $X$ is intermediate between $X_1$ and $X_2$ if there exist the
embeddings $X_1 \cap X_2 \rightarrow X \rightarrow X_1+X_2$.\\
(ii) Suppose $s\in (0,\infty)$, $p \in [0,\infty)$ and $q \in [1,\infty]$. Moreover, let $m$
be the smallest integer larger than $s$. We define the Besov space $B^{s;p,q}(\Om)$ to be
the intermediate space between $L^p(\Om)$ and $W^{m,p}(\Om)$ corresponding to
the $J$ method of interpolation with $\theta=s/m$, specifically:
\begin{equation*}
B^{s;p,q}(\Om)=\left( L^p(\Om), W^{m,p}(\Om)\right)_{s/m,q;J}.
\end{equation*}
\end{definition}
\begin{remark}\label{rem:Besov}
(i) The spaces $X_1 \cap X_2$ and $X_1+X_2=\{u=u_1+u_2:u_1 \in X_1, u_2 \in X_2\}$ mentioned in
Definition \ref{def:Besov} are Banach spaces with respect to the norms
$\|u\|_{X_1\cap X_2}=\text{max}\{\|u\|_{X_1}, \|u\|_{X_2}\}$ and
$\|u\|_{X_1+ X_2}= \text{inf}\{\|u_1\|_{X_1}+\|u_2\|_{X_2}: u=u_1+u_2,u_1 \in X_1, u_2 \in X_2\}$
(cf. \cite[7.7]{AdamsFournier:2003}). \\
(ii) We have the following embeddings (cf. \cite[7.33]{AdamsFournier:2003}):
\begin{align*}
B^{m;p,p}(\Om)&\hookrightarrow W^{m,p}(\Omega)\hookrightarrow B^{m;p,2}(\Om)& \quad &\text{for } 1<p\leq2, \\
B^{m;p,2}(\Om)&\hookrightarrow W^{m,p}(\Omega)\hookrightarrow B^{m;p,p}(\Om)& \quad &\text{for } 2\leq p\leq\infty.
\end{align*}
Therefore, if $p=2$, we have $B^{m,2,2}(\Om)=W^{m,2}(\Om)=H^m(\Om)$.
\end{remark}

\section{The linear problem - semigroup framework}
Before turning to the nonlinear analysis, we consider the linearization of
\eqref{fourthordereqnograd} in a generalized abstract form.

We investigate the abstract linear partial differential equation
\begin{equation}\label{absfourthordereq}
\left(-a\mathcal{A} - \partial_t \right)
\left(\psi_{tt}(t)+c^2\mathcal{A}\psi(t) + b\mathcal{A}\psi_t(t) \right)=f(t)
\end{equation}
defined on the Hilbert space $\h$ with initial conditions
\begin{equation}\label{initialcond}
\psi(0)=\psi_0, \quad \psi_t(0)=\psi_1, \quad \psi_{tt}(0)=\psi_2.
\end{equation}
\subsection{Semigroup methods for the homogeneous equation}
We are first going to treat the homogeneous version of \eqref{absfourthordereq}, i.e.
\begin{equation}\label{abshomfourthordereq}
\left(-a\mathcal{A} - \partial_t \right)
\left(\psi_{tt}(t)+c^2\mathcal{A}\psi(t) + b\mathcal{A}\psi_t(t) \right)=0
\end{equation}
defined on $\h$ with initial conditions \eqref{initialcond}.
\begin{notation} Let $x=(x_1,x_2,x_3)^T$, $y=(y_1,y_2,y_3)^T$. We introduce the spaces
\begin{align}
H_0 &:= \h \times \h \times \h,\\
H_1 &:= \D(\A^{1/2}) \times \h \times \h,\\
H_2 &:= \DA \times \h \times \h,
\end{align}
equipped with the inner products
\begin{align}
\spr{x}{y}_{H_0}&:=\spr{x_1}{y_1}_\h + \spr{x_2}{y_2}_\h + \spr{x_3}{y_3}_\h,~~& x,y \in H_0,\\
\spr{x}{y}_{H_1}&:=\spr{\A^{1/2}x_1}{\A^{1/2}y_1}_\h + \spr{x_2}{y_2}_\h + \spr{x_3}{y_3}_\h,~~&x,y \in H_1,\\
\spr{x}{y}_{H_2}&:= \frac{\alpha^2 b^2}{4}\spr{\A x_1}{\A y_1}_\h + \spr{x_2}{y_2}_\h + \spr{x_3}{y_3}_\h, ~~ &x,y \in H_2,\label{sprh1}
\end{align}
respectively. The respective norms are given by $\|x\|_{H_i}=\sqrt{\spr{x}{x}_{H_i}},~i\in\{1,2,3\}$.
The constant $\alpha >0$ in \eqref{sprh1} will be determined later.
\end{notation}
The partial differential equation \eqref{abshomfourthordereq} can be written as a first-order system of the form
\begin{equation}\label{abstractODE}
\Psi_t(t)=A\Psi(t), \qquad t>0
\end{equation}
with initial conditions
\begin{equation}\label{initialcondmatrix}
\Psi(0) = \Psi_0 = \begin{pmatrix} \psi_0 \\ \psi_1 \\ \psi_2 + c^2 \mathcal{A}\psi_0 + b \mathcal{A}\psi_1 \end{pmatrix} \in H_i,~~i\in\{1,2,3\}
\end{equation}
if we choose
\begin{equation}\label{generator1}
A= \begin{pmatrix}
0 & I & 0 \\ -c^2 \mathcal{A} & -b\mathcal{A}  & I \\  0 & 0 & -a \mathcal{A}
\end{pmatrix}, \qquad \D(A)=\DA \times \DA \times \DA
\end{equation}
and
\begin{equation}\label{psi:big}
\Psi(t)= \begin{pmatrix} \psi(t) \\ \psi_t(t) \\ \psi_{tt}(t) +c^2 \mathcal{A}\psi(t) +b\mathcal{A}\psi_t(t)\end{pmatrix}.
\end{equation}
\subsubsection{Generation on $H_1$}
We decompose $A$, $A=\tilde{A}_1+B_1$, where
\begin{equation}\label{decomposition:triggiani}
\tilde{A}_1=\begin{pmatrix} 0 & I & 0 \\ -c^2\A & -b\A & 0 \\ 0 & 0 & -a\A \end{pmatrix}
\qquad \text{and} \qquad
B_1= \begin{pmatrix} 0 & 0 & 0 \\ 0 & 0 & I \\ 0 & 0 & 0 \end{pmatrix}.
\end{equation}
Operators of the type
\begin{equation*}
\begin{pmatrix}
0 & I \\ -c^2\A & -b\A
\end{pmatrix}
\end{equation*}
have been extensively investigated, e.g., in \cite{ChenRussell:1982}, \cite{ChenTriggiani:1989} and \cite{LiangXiao:1997}
and we will modify some arguments used therein in order to show that $A$ given by \eqref{generator1} generates an analytic
semigroup on $H_1$.

First, note that $\tilde{A}_1$ acting on $H_1$ is essentially equivalent to the operator
\begin{equation}\label{hatA}
\hat{A}_1=\begin{pmatrix}
0 & c\A^{1/2} & 0 \\ -c\A^{1/2} & -b\A & 0 \\ 0 & 0 & -a\A
\end{pmatrix},
\qquad \D(\hat{A}_1)=\D(\A^{1/2}) \times \DA \times \DA
\end{equation}
acting on $H_0$ (cf. \cite{ChenRussell:1982}, \cite{LiangXiao:1997}).
\begin{lemma}
The operator $\hat{A}_1$ given by \eqref{hatA} is dissipative on $H_0$, i.e.
\begin{equation*}
\norm{(\lam I - \hat{A}_1)x}_{H_0} \geq \lam \norm{x}_{H_0} \qquad \text{for all } x\in\D(\hat{A}_1) \text{ and } \lam >0.
\end{equation*}
\end{lemma}
\begin{proof} We proceed analogously to \cite{ChenRussell:1982}. 
Let $x=(x_1,x_2,x_3)^T \in \D(\hat{A}_1)$. The straightforward estimate
\begin{align*}
\norm{(\lam I - \hat{A}_1)x}_{H_0}^2 =~&\lam^2 \left(\hnorm{x_1}^2+\hnorm{x_2}^2 + \hnorm{x_3}^2\right)+\hnorm{c\A^{1/2} x_1}^2\\
&+\spr{b\A x_2}{c\A^{1/2} x_1}_\h+\spr{c\A^{1/2} x_1}{b\A x_2}_\h +\hnorm{b\A x_2}^2\\
&+\hnorm{c\A^{1/2} x_2}^2 +2 \lam \hnorm{b\A x_2}^2 + 2\lam \hnorm{a\A x_3}^2 + \hnorm{a\A x_3}^2\\
\geq&~\lam^2 \norm{x}_{H_0}^2 + \hnorm{c\A^{1/2} x_1 + b\A x_2}^2\\
\geq&~\lam^2 \norm{x}_{H_0}^2
\end{align*}
yields the desired result.
\end{proof}
\begin{proposition}
The operator $\tilde{A}_1$ generates a strongly continuous semigroup of contractions on $H_1$.
\end{proposition}
\begin{proof}
Observe that $\hat{A}_1$ is densely defined on $H_0$. Furthermore, there exists a $\lam_0 >0$ such that
the range of $\lam_0 I -\hat{A}_1$ equals $H_0$. On the strength of the Lumer-Phillips Theorem we conclude that
$\hat{A}_1$ generates a strongly continuous semigroup of contractions on $H_0$ and thus $\tilde{A}_1$ generates a strongly continuous
semigroup of contractions on $H_1$.
\end{proof}
The resolvent $R(\lam,\tilde{A}_1)$ of $\tilde{A}_1$ is explicitly given by
\begin{equation}\label{tildeAresolvent}
R(\lam,\tilde{A}_1)=\begin{pmatrix}
V(\lam)^{-1}(\lam I+b\A) & V(\lam)^{-1} & 0 \\ -V(\lam)^{-1}c^2\A & \lam V(\lam)^{-1} & 0 \\ 0 & 0 & (\lam I +a\A)^{-1}
\end{pmatrix}
\end{equation}
where we abbreviate $V(\lam)=\lam^2 I +b\lam \A +c^2 \A$. Note that $V(\lam)$ and $\A$ commute.
\begin{lemma}
\label{lem:triggiani} We have the following uniform bounds for all $\lam$ with $\re(\lam)>0$:
\begin{align}
\norm{\lam^2 V(\lam)^{-1}} &\leq C_{\underline{\mu},b,c} :=
\begin{cases}
1 & \mbox{if } \frac{b^2}{2c^2} \underline{\mu}\geq 1,\\
\left[\frac{b^2}{c^2}\underline{\mu} \left(1-\frac{b^2}{4c^2}\underline{\mu}\right)\right]^{-1/2} & \mbox{if } \frac{b^2 }{2c^2}\underline{\mu}<1,
\end{cases}\\
\norm{b\lam \A V(\lam)^{-1}} &\leq 1. \label{lamv}
\end{align}
where $\underline{\mu}:=\min \sigma(\A)$
\end{lemma}
\begin{proof}
Choose $A=c^2\A$, $\rho=\frac{b}{2c^2}$ and $\alpha =1$ in \cite[Proposition 3.1]{ChenTriggiani:1989}.
\end{proof}
\begin{proposition}
The operator $\tilde{A}_1$ generates an analytic semigroup on $H_1$.
\end{proposition}
\begin{proof}
As we already know that $\tilde{A}_1$ generates a strongly continuous semigroup of contractions on $H_1$,
it remains to show that there exists a constant $M>0$ such that
\begin{equation}\label{resolventest}
\norm{R(\lam,\tilde{A}_1)}_{L(H_1)}\leq \frac{M}{|\lam|} \qquad \text{for all } \lam \in \C \text{ with } \re(\lam)>0.
\end{equation}
Recall the explicit representation of the resolvent \eqref{tildeAresolvent}. Let $\lam \in \C$ with $\re(\lam) >0$ and
$x=(x_1,x_2,x_3)^T \in H_1$. By Lemma \ref{lem:triggiani} and \eqref{Poincare} we have that

\begin{align*}
\norm{\lam R(\lam,\tilde{A}_1)x}_{H_1}^2  = &~\hnorm{\lam^2V(\lam)^{-1} \A^{1/2}x_1+b\lam\A V(\lam)^{-1} \A^{1/2}x_1+\lam V(\lam)^{-1}\A^{1/2}x_2}^2\\
&~\hnorm{-\lam c^2V(\lam)\A x_1 + \lam^2 V(\lam)^{-1}x_2}^2 + \hnorm{\lam(\lam I + a\A)^{-1}x_3}^2\\
\leq &~3 \hnorm{\lam^2 V(\lam)^{-1}\A^{1/2}x_1}^2+3\hnorm{b\lam \A V(\lam)^{-1}\A^{1/2}x_1}^2\\
&~+3\hnorm{\lam V(\lam)^{-1}\A^{1/2}x_2}^2+2\hnorm{\lam c^2 \A^{1/2}V(\lam)^{-1}\A^{1/2}x_1}^2\\
&+2\hnorm{\lam^2V(\lam)^{-1}x_2}^2 + \hnorm{\lam(\lam I +a\A)x_3}^2\\
\leq &~3 \hnorm{\lam^2 V(\lam)^{-1}\A^{1/2}x_1}^2 + 3 \hnorm{b\lam \A V(\lam)^{-1}\A^{1/2}x_1}^2 \\
&~+3C_{\D(\A)\hookrightarrow \h}^2 b^{-2} \hnorm{b\lam \A V(\lam)^{-1} x_2}^2\\
&~+2C_{\D(\A^{1/2})\hookrightarrow\h}^2 c^4 b^{-2} \hnorm{b \lam \A V(\lam)^{-1}\A^{1/2}x_1}^2\\
&~+2\hnorm{\lam^2V(\lam)^{-1}x_2}^2+ \hnorm{\lam(\lam I +a\A)x_3}^2\\
\leq&~M^2 \norm{x}_{H_1}^2
\end{align*}
where $M:=\max\left(3C_{\underline{\mu},b,c}+\frac{2c^4}{b^2} C_{\D(\A^{1/2})\hookrightarrow \h}^2 +3,~\frac{3}{b^2} C_{\D(\A)\hookrightarrow \h}^2+2C_{\underline{\mu},b,c},1\right)^{1/2}$.
\end{proof}
\begin{remark} \label{rem:b} Note that the assumption $b>0$ is essential to establish the uniform bound \eqref{resolventest}. In case $b=0$
the strongly continuous semigroup of contractions generated by $\tilde{A}_1$ on $H_1$ is not analytic (see also Remark \ref{rem:spectrum}).
\end{remark}
\begin{theorem}
The operator $A: \D(A) \rightarrow H_1$ given by \eqref{generator1} generates an analytic semigroup T(t) on $H_1$.
\end{theorem}
\begin{proof}
Note that $B_1$ from \eqref{decomposition:triggiani} is a bounded operator on $H_1$. Therefore, the result follows at once from
the perturbation theorem for analytic semigroups.
\end{proof}

\subsubsection{Generation on $H_2$}
Next, we decompose $A$, $A=\tilde{A} + B$, where
\begin{equation}\label{decomposition}
\tilde{A}=\begin{pmatrix} 0 & 0 & 0 \\ 0 & -b\mathcal{A} & 0 \\ 0 & 0 & -a\mathcal{A} \end{pmatrix}
\qquad \text{and} \qquad
B= \begin{pmatrix} 0 & I & 0 \\ -c^2\mathcal{A} & 0 & I \\ 0 & 0 & 0 \end{pmatrix}.
\end{equation}
\begin{lemma}
The linear operator $\tilde{A}$ is the generator of a bounded analytic semigroup on $H_2$.
\end{lemma}
\begin{proof}
Note that $\tilde{A}$ is self-adjoint on $H_2$ and its spectrum is given by
$\sigma(\tilde{A})=\{0,-a\mu_i,-b\mu_i: \mu_i \in \sigma(\mathcal{A})\} \subset \R_0^-:=\{x\in \R: x\leq 0\}$.
The result follows at once by using Lemma \ref{lem:normalop}. 
\end{proof}

\begin{lemma} We have
\begin{equation}\label{relboundest}
\|Bx\|_{H_2} \leq \frac{\alpha}{2} \|\tilde{A}x\|_{H_2} + \sqrt{2}\max\left\{\frac{2c^2}{\alpha b},1\right\} \|x\|_{H_2}   \qquad \text{for all } x \in \mathcal{D}(A).
\end{equation}
In particular, $B$ is $\tilde{A}$-bounded with $\tilde{A}$-bound $\alpha_0 = \frac{\alpha}{2} <\alpha$.
\end{lemma}
\begin{proof}
Note that $\|\tilde{A}x\|_{H_2}^2 = b^2 \|\mathcal{A}x_2\|_\mathcal{H}^2 + a^2 \|\mathcal{A}x_3\|_\mathcal{H}^2$.
Moreover,
\begin{align}\nonumber
\|Bx\|_{H_2}^2 &= \left(\frac{\alpha b}{2}\right)^2\|\mathcal{A}x_2\|_\mathcal{H}^2 + \|-c^2\mathcal{A}x_1+x_3\|_\mathcal{H}^2 \\ \nonumber
& \leq \left(\frac{\alpha b}{2}\right)^2 \|\mathcal{A}x_2\|_\mathcal{H}^2+ \left(c^2 \|\mathcal{A}x_1\|_\mathcal{H} + \|x_3\|_\mathcal{H}\right)^2 \nonumber\\
& \leq \left(\frac{\alpha b}{2}\right)^2\|\mathcal{A}x_2\|_\mathcal{H}^2 +2c^4\|\mathcal{A}x_1\|_\mathcal{H}^2 +2 \|x_3\|_\mathcal{H}^2\nonumber \\
& \leq \left(\frac{\alpha}{2}\right)^2 \|\tilde{A}x\|_{H_2}^2 + 2\max\left\{\frac{4c^4}{\alpha^2 b^2},1\right\} \|x\|_{H_2}^2\nonumber\\
& \leq \left( \frac{\alpha}{2} \|\tilde{A}x\|_{H_2} + \sqrt{2}\max\left\{\frac{2c^2}{\alpha b},1\right\} \|x\|_{H_2}\right)^2\nonumber
\end{align}
and therefore \eqref{relboundest} as stated.
\end{proof}\noindent
\begin{theorem}
The operator $A: \mathcal{D}(A) \rightarrow H_2$ given by \eqref{generator1} generates an analytic semigroup $S(t)$ on $H_2$.
\end{theorem}
\begin{proof}
We have to choose $\alpha$ in $\spr{.}{..}_{H_2}$ small enough such that the sum $\tilde{A}_1+B_1$ generates an
analytic semigroup (according to Proposition \ref{prop:perturb}).
\end{proof}

\subsubsection{Exponential decay for the homogeneous equation}
The previous results enable us now to show exponential decay of solutions in the homogeneous case.
\begin{notation}We introduce the energy functionals
\begin{align}
E_1[\psi](t) &:= \|\A^{1/2}\psi(t)\|_\h^2 + \|\psi_t(t)\|_\h^2 + \|\psi_{tt}(t)+c^2\A\psi(t) +b\A\psi_t(t)\|_\h^2,\\
E_2[\psi](t) &:= \|\A\psi(t)\|_\h^2 + \|\psi_t(t)\|_\h^2 + \|\psi_{tt}(t)+c^2\A\psi(t) +b\A\psi_t(t)\|_\h^2.
\end{align}
\end{notation}
Our aim is to show exponential decay for the energies $E_1[\psi]$ and $E_2[\psi]$.
\begin{lemma}\label{lem:omega}
The spectral bound $s(A)=\sup \{\re(\lam):\lam \in \sigma(A)\}$ of $A$ is given by
\begin{equation*}
s(A)=-\min \left\{a\underline{\mu}, \frac{b\underline{\mu}}{2},\frac{c^2}{b}\right\}.
\end{equation*}
\end{lemma}
\begin{proof}The spectrum of $A$ is given by $\sigma(A)=\{\kappa_1(\mu_i), \kappa_2(\mu_i), \kappa_3(\mu_i): \mu_i \in \sigma(\mathcal{A})\}$
where
\begin{equation*}
\kappa_1(\mu_i)=-a\mu_i \quad \text{and} \quad \kappa_{2,3}(\mu_i) = \frac{1}{2}\left(-b\mu_i\pm\left(b^2 \mu_i^2-4c^2\mu_i\right)^{1/2}\right).
\end{equation*}
First, $\kappa_1(\mu_i) \in \R$ is bounded from above by $\kappa_1(\underline{\mu})=-a\underline{\mu}$. 

If $\mu_i < \frac{4c^2}{b^2}$, then $\kappa_{2,3}(\mu_i) \in \C$ with real part $\re(\kappa_{2,3}(\mu_i))=-\frac{b\mu_i}{2}$ which is bounded from above by $\re(\kappa_{2,3}(\underline{\mu}))=-\frac{b\underline{\mu}}{2}<0$.

If $\mu_i \geq \frac{4c^2}{b^2}$, then $\kappa_{2,3} \in \R$.
We have that $\kappa_2(\mu_i)$ is increasing for $\mu_i \geq \frac{4c^2}{b^2}$ and $\lim_{\mu_i \rightarrow \infty} \kappa_2(\mu_i)=-\frac{c^2}{b}$.
Furthermore, $\kappa_3(\mu_i)$ is decreasing for $\mu_i \geq \frac{4c^2}{b^2}$ and therefore
$\{\kappa_3(\mu_i) \cap \R: \mu_i \in \sigma(\mathcal{A})\}$ is bounded from above
by $\kappa_3(\frac{4c^2}{b^2})=-\frac{2c^2}{b}$.

Combining the upper bounds for $\re(\kappa_n(\mu_i))$, $n\in\{1,2,3\}$ leads to the desired spectral bound.
\end{proof}
\begin{remark} \label{rem:spectrum} Note that for Lemma \ref{lem:omega} the assumption $b>0$ is essential. If $b=0$, the spectrum of $A$ is
given by $\sigma(A)=\{-a\mu_i, \pm \I c\mu_i: \mu_i \in \sigma(\A) \}$. Hence, in this case, $A$ is not a sectorial operator and can thus not be
the generator of an analytic semigroup (cf. Theorem \ref{thm:equivanalytic}).
\end{remark}

\begin{theorem}\label{h2:decay}
There exist positive constants $M_1, M_2, \om_1, \om_2$ such that
\begin{equation*}
E_1[\psi](t) \leq M_1 \E^{-\om_1 t} E_1[\psi](0) \quad \text{ and } \quad
E_2[\psi](t) \leq M_2 \E^{-\om_2 t} E_2[\psi](0).
\end{equation*}
\end{theorem}
\begin{proof}
As $s(A)<0$, the analytic semigroups $T(t)$ and $S(t)$ generated by $A$ in \eqref{generator1} on the spaces $H_1$ respectively $H_2$ are
uniformly exponentially stable, i.e. there exist constants $M_1, \tilde{M}_2 >1$ and $\om_1, \om_2>0$ such that
$\norm{T(t)x}_{H_1} \leq M_1 \E^{-\om_1 t}\norm{x}_{H_1}$, $x\in H_1$ and $\norm{S(t)x}_{H_2} \leq \tilde{M}_2 \E^{-\om_2 t}\norm{x}_{H_2}$, $x\in H_2$.
Setting $x=\Psi(0)$ and rescaling the estimate in case of $H_2$ yields the claim.
\end{proof}

\subsubsection{Solutions of the homogeneous equation} We now consider the homogeneous initial boundary value problem
\begin{equation}\label{abbreq:hom}
\begin{cases}
(a\Delta - \partial_t)\left(\psi_{tt}-c^2\Delta\psi - b\Delta\psi_t \right)=0  & \text{ in } \Omega \times (0,T],
\\
(\psi,\psi_t,\psi_{tt})=(\psi_0, \psi_1, \psi_2) & \text{ on } \Omega \times \{t=0\},
\\
\psi =0  & \text{ on } \partial\Omega \times [0,T],
.\\
\end{cases}
\end{equation}
which we represented as \eqref{abstractODE}.
\begin{definition}[{\cite[4.4.1]{Pazy:1983}}] Suppose $X$ is a Banach space and $\mathbf{A}: \mathcal{D}(\mathbf{A}) \subset X \rightarrow X$ is a linear operator.
We call an $X$-valued function $u(t)$ a solution of the Cauchy problem with initial data $x$,
\begin{equation}\label{cauchy:hom}
\begin{aligned}
u_t (t)&=\mathbf{A}u(t),~~~~~~~ t>0,\\
u(0)&=x,
\end{aligned}
\end{equation}
if $u(t)$ is continuous for $t\geq 0$, continuously differentiable and $u(t) \in \D(\mathbf{A})$ for $t>0$ and \eqref{cauchy:hom} is satisfied.
\end{definition}
Recall that, if $\mathbf{A}$ is the infinitesimal generator of an analytic semigroup, the initial value problem
\eqref{cauchy:hom} has a unique solution for every $x\in X$ (\cite[Corollary 4.1.5]{Pazy:1983}).

\begin{corollary} \label{cor:hom} The homogeneous initial boundary value problem \eqref{abbreq:hom}
has a unique solution
\begin{equation*}
\psi \in  C^1(0,T;\hii \cap \hio) \cap C^2(0,T;L^2(\Om))
\end{equation*}
for all $T>0$ provided $\psi_0 \in H_0^1(\Om)$, $\psi_1 \in \lii$ and $\psi_2 -b\Delta \psi_1-c^2\Delta \psi \in \lii$.

In particular, it has a unique solution of this regularity, if $\psi_0 \in \hii \cap H_0^1(\Om)$, $\psi_1 \in \hii \cap H_0^1(\Om)$ and $\psi_2 \in \lii$.
\end{corollary}
\begin{proof}
Recall that, if $\mathbf{A}$ is the infinitesimal generator of an analytic semigroup, the initial value problem
\eqref{cauchy:hom} has a unique solution for every $x\in X$ (\cite[Corollary 4.1.5]{Pazy:1983}).
As $A$ given by \eqref{generator1} is analytic on $H_i$, $i\in\{1,2\}$ we conclude that the initial value problem \eqref{abstractODE} with initial conditions \eqref{initialcondmatrix}
has a unique solution for all $\Psi_0 \in H_i$, $i\in\{1,2\}$, i.e. there exists a unique function $\Psi \in H_i$, $i\in\{1,2\}$
such that $\Psi$ is continuous for $t\geq 0$, continuously differentiable, $\Psi \in \D(A)$ for $t>0$ and \eqref{abstractODE}
is satisfied. This yields the first claim.

For the second claim note that $\psi_0 \in \hii \cap H_0^1(\Om)$, $\psi_1 \in \hii \cap H_0^1(\Om)$ and $\psi_2 \in \lii$ implies $\psi_0 \in H_0^1(\Om)$, $\psi_1 \in \lii$ and $\psi_2 -b\Delta \psi_1-c^2\Delta \psi \in \lii$.
\end{proof}

\subsection{Semigroup methods for the inhomogeneous equation} In this section we consider the inhomogeneous initial boundary value problem
\begin{equation}\label{abbreq:inhom}
\begin{cases}
(a\Delta - \partial_t)\left(\psi_{tt}-c^2\Delta\psi - b\Delta\psi_t \right)=f  & \text{ in } \Omega \times (0,T],
\\
(\psi,\psi_t,\psi_{tt})=(\psi_0, \psi_1, \psi_2) & \text{ on } \Omega \times \{t=0\},
\\
\psi =0  & \text{ on } \partial\Omega \times [0,T],
\\
\end{cases}
\end{equation}
where $f: \Om \times (0,T] \rightarrow\R$ is given. We represent it as an inhomogeneous abstract ordinary differential equation of the form
\begin{equation}\label{inhomeq}
\Psi_t(t)=A\Psi(t) + F(t), \qquad t \in (0,T)
\end{equation}
with the initial conditions \eqref{initialcond}, where $A$ and $\Psi$ are given by \eqref{generator1} and \eqref{psi:big}, respectively and
$F(t)=(0,0,f(t))^T$.
\begin{definition}[{\cite[4.2.1]{Pazy:1983}}] Suppose $X$ is a Banach space and consider
the inhomogeneous initial value problem $u_t(t)=\mathbf{A}u(t)+\mathbf{F}(t)$, $t\in(0,T)$ with initial condition $u(0)=u_0$.

A function $u: [0,T) \rightarrow X$ is a (classical) solution of the inhomogeneous initial value problem on $[0,T)$ if $u$ is continuous on $[0,T)$,
continuously differentiable on $(0,T)$, $u(t) \in \D(\mathbf{A})$ for $0<t<T$ and $u_t(t)=\mathbf{A}u(t)+\mathbf{F}(t)$ is satisfied on $[0,T)$.
\end{definition}

\begin{corollary}
Let $\psi_0 \in H_0^1(\Om)$, $\psi_1 \in \lii$ and $\psi_{tt}-b\Delta\psi_t-c^2\Delta\psi \in \lii$. Furthermore, suppose $f\in L^1(0,T;\lii)$
is locally H{\"o}lder-continuous on $(0,T]$. Then the initial boundary value problem \eqref{abbreq:inhom} has a unique solution
\begin{equation*}
\psi \in  C^1(0,T;\hii\cap \hio) \cap C^2(0,T;L^2(\Om)).
\end{equation*}
In particular, it has a unique solution of this regularity, if $\psi_0 \in \hii \cap H_0^1(\Om)$, $\psi_1 \in \hii \cap H_0^1(\Om)$ and $\psi_2 \in \lii$.
\end{corollary}
\begin{proof}
The result follows analogously to the one Corollary \ref{cor:hom} by applying \cite[Corollary 4.3.3]{Pazy:1983}.
\end{proof}

\section{Energy estimates}
In this section, we derive energy estimates which will enable us to prove global
existence of solutions in Section \ref{sec:nonlinear}.
Again, we consider the equation
\begin{equation*}
(\partial_t+a \A)(\psi_{tt}+c^2\A\psi+b\A\psi_t)=-\sigma (\psi_t^2)_{tt}
\end{equation*}
or, equivalently,
\begin{equation*}
(\partial_{tt} + c^2 \A + b \A\partial_t)(\psi_t+a\A\psi)=-\sigma (\psi_t^2)_{tt}
\end{equation*}
i.e.
\begin{equation*}
D_w w = f=-\sigma (\psi_t^2)_{tt} \qquad \text{ and } \qquad D_h\psi=w
\end{equation*}
where $D_h=\partial_{t} + a\A$ and $D_w=\partial_{tt} + c^2 \A + b \A\partial_t$.
\begin{remark}
In order to interchange the order of differentiation, we need to assume that the following
estimates only hold for sufficiently smooth solutions. But in fact, by using density arguments,
this restriction can finally be removed.
\end{remark}
\begin{notation}
We introduce the energy functionals
\begin{align}
\cE_0[w](t)&:=\frac{1}{2}\left( \hnorm{w_{tt}(t)}^2+\hnorm{\A^{1/2}w_t(t)}^2+\hnorm{\A w(t)}^2\right)\label{energy0:westervelt},\\
\cE[w](t)&:=\frac{1}{2}\left( \hnorm{\A^{1/2}w_{tt}(t)}^2+\hnorm{\A^{1/2}w_t(t)}^2+\hnorm{\A w(t)}^2\right),\label{energy:westervelt}\\
E[\psi](t)&:=\frac{1}{2} \left(\hnorm{\A^{1/2}\psi_{ttt}(t)}^2+\hnorm{\A\psi_{tt}(t)}^2+\hnorm{\A\psi_{t}(t)}^2\right), \label{energy:psi}
\end{align}
as well as the sum of \eqref{energy:westervelt} and \eqref{energy:psi},
\begin{equation}
\Lambda(t):= E[\psi](t) + \cE[w](t).\label{energy:sum}
\end{equation}
\end{notation}
\begin{lemma}\label{lem:Prop51} For a solution $w$ of $D_w w= f$ with $f_t \in L^2(0,T;\h)$ we have the estimate
\begin{equation}\label{ieq:Th51iii}
\begin{aligned}
&\frac{1}{2}\int_0^t\hnorm{w_{ttt}(\tau)}^2\, d\tau+
\frac{b}{2}\hnorm{\A^{1/2}w_{tt}(\tau)}^2 \Big|_0^t\\
&\leq c^4\int_0^t \hnorm{\A w_t(\tau)}^2 \, d\tau
+ \int_0^t \hnorm{f_t(\tau)}^2\, d\tau.
\end{aligned}
\end{equation}
\end{lemma}
\begin{proof}
Differentiating $D_w w =f$ with respect to time, taking
inner products in $\h$ with $w_{ttt}$ and then integrating with respect to time we get
\begin{equation*}
\begin{aligned}
&\int_0^t \spr{w_{ttt}(\tau)}{w_{ttt}(\tau)}_\h\, d\tau
+c^2\int_0^t\spr{\A w_t(\tau)}{w_{ttt}(\tau)}_\h \,d\tau \\
&+ b\int_0^t\spr{\A w_{tt}(\tau)}{w_{ttt}(\tau)}_\h \,d\tau
= \int_0^t\spr{f_t(\tau)}{w_{ttt}(\tau)}_\h\,d\tau.
\end{aligned}
\end{equation*}
Performing integration by parts, we obtain
\begin{equation}\label{ibp:proofwestervelt}
\int_0^t \hnorm{w_{ttt}(\tau)}^2\, d\tau + \frac{b}{2}\hnorm{\A^{1/2}w_{tt}(\tau)}^2 \Big|_0^t
= \int_0^t\spr{f_t(\tau)- c^2\A w_t(\tau)}{w_{ttt}(\tau)}_\h \,d\tau.
\end{equation}
Estimating the right hand side yields
\begin{align*}
&\nonumber\int_0^t\spr{f_t(\tau)-c^2\A w_t(\tau)}{w_{ttt}(\tau)}_\h \, d\tau\\
&\leq \int_0^t \hnorm{f_t(\tau)-c^2\A w_t(\tau)} \hnorm{w_{ttt}(\tau)} \, d\tau \\
&\leq \nonumber \frac{1}{2}\int_0^t\hnorm{w_{ttt}(\tau)}^2 \, d\tau
+\frac{1}{2} \int_0^t\hnorm{f_t(\tau)-c^2\A w_t(\tau)}^2 \, d\tau \\
&\leq \nonumber \frac{1}{2}\int_0^t\norm{w_{ttt}(\tau)}_\h^2\, d\tau
+c^4\int_0^t \hnorm{\A w_t(\tau)}^2 \, d\tau + \int_0^t\hnorm{f_t(\tau)}^2 \, d\tau
\end{align*}
which, together with \eqref{ibp:proofwestervelt}, implies the desired estimate \eqref{ieq:Th51iii}.
\end{proof}
\begin{lemma}The estimate
\begin{equation}\label{est1}
\begin{aligned}
&\cE[w](t)+ \check{b}\int_0^t\left\{
\hnorm{w_{ttt}(\tau)}^2+\hnorm{\A^{1/2}w_{tt}(\tau)}^2+\hnorm{\A w_t(\tau)}^2+\hnorm{\A w(\tau)}^2\right\}d\tau\\
&\leq C \left(\cE[w](0)+\int_0^t\left\{
 \hnorm{(\psi_t^2)_{ttt}(\tau)}^2+\hnorm{(\psi_t^2)_{tt}(\tau)}^2\right\}d\tau\right)
\end{aligned}
\end{equation}
holds with $\check{b}$ sufficiently small and $C$ sufficiently large.
\end{lemma}
\begin{proof}
We recall Proposition 3 in \cite{KL09} with $\alpha\equiv1$, i.e. for the energy $\cE_0[w](t)$ in \eqref{energy0:westervelt}
we have the estimate
\begin{equation}\label{est:prop3}
\begin{aligned}
&\cE_0[w](t)+ \hat{b}\int_0^t\left\{
\hnorm{\A^{1/2}w_{tt}(\tau)}^2+\hnorm{\A^{1/2}w_t(\tau)}^2+\hnorm{\A w(\tau)}^2\right\}d\tau\\
&\leq \tilde{C} \left(\cE_0[w](0)+\int_0^t\left\{
 \hnorm{\A^{-1/2}f_t(\tau)}^2+\hnorm{f}^2\right\}d\tau\right)
 \end{aligned}
\end{equation}
for $\hat{b}$ sufficiently small and $\tilde{C}$ sufficiently large. 
Next, we use \eqref{ieq:Th51iii}
\begin{equation}\label{est:Th51iii}
\begin{aligned}
&\frac{\lambda b}{2}\hnorm{\A^{1/2}w_{tt}(t)}^2
+\frac{\lambda}{2}\int_0^t \hnorm{w_{ttt}(\tau)}^2\,d\tau
-\lambda c^4\int_0^t \hnorm{\A w_t(\tau)}^2\,d\tau\\
&\leq \frac{\lambda b}{2}\hnorm{\A^{1/2}w_{tt}(0)}^2+
\lambda\int_0^t \hnorm{f_t(\tau)}^2\,d\tau
\end{aligned}
\end{equation}
and multiply it with
a sufficiently small constant
$\lambda$,
$
\lambda\leq\tfrac{\hat{b}b}{6c^4\max\big\{C_{\D(\A^{1/2})\hookrightarrow \h},c^4\big\}}.
$
Using the estimate (from $D_w w=f$) yields
\begin{equation}\label{est:Th51ii}
\begin{aligned}
\hnorm{\A w_t}^2=\frac{1}{b^2}\hnorm{w_{tt}+c^2\A w-f}^2\leq
\frac{1}{b^2}\left(3\hnorm{w_{tt}}^2+3c^4\hnorm{\A w}^2+4\hnorm{f}^2\right)
\end{aligned}
\end{equation}
which implies that, on the left hand side of \eqref{est:prop3}, we have
\begin{equation*}
\begin{aligned}
&\int_0^t\left\{\hnorm{\A^{1/2}w_{tt}(\tau)}^2+\hnorm{\A w(\tau)}^2\right\}\,d\tau\\
&\geq c_0 \int_0^t\left\{
\hnorm{\A^{1/2}w_{tt}(\tau)}^2+\hnorm{\A w(\tau)}^2+\hnorm{\A w_t(\tau)}^2\right\}\,d\tau
-C_0\int_0^t\hnorm{f(\tau)}^2\, d\tau
\end{aligned}
\end{equation*}
where $c_0=\left(\max\left\{1+3b^{-2}C_{\D(\A)\hookrightarrow \h}\, C_{\D(\A^{1/2})\hookrightarrow \h}, 1+3c^4b^{-2}\,C_{\D(\A)\hookrightarrow \h}\right\}\right)^{-1}$
and $C_0=3b^{-2}c_0$ and that in \eqref{est:Th51iii} we get (choice of $\lambda$)
\begin{equation*}
\lambda c^4\int_0^t \hnorm{\A w_t(\tau)}^2\,d\tau
\leq  \frac{\hat{b}}{2}
\int_0^t\left\{\hnorm{\A^{1/2}w_{tt}(\tau)}^2+\hnorm{\A w(\tau)}^2 + \hnorm{f(\tau)}^2\right\}d\tau\,.
\end{equation*}
Adding \eqref{est:prop3} and \eqref{est:Th51ii} and using \eqref{Poincare} for $\hnorm{\A^{-1/2}f}$, we obtain \eqref{est1}
with $\check{b}$ sufficiently small and $\check{C}$ sufficiently large.
\end{proof}

Now we use the following energy identity for the heat equation:
\begin{equation}\label{id:heat1}
\begin{aligned}
&\int_0^t \hnorm{D_h v(\tau)}^2\, d\tau = \int_0^t \hnorm{v_t(\tau)+a\A v}^2\, d\tau \\
&=a\hnorm{\A^{1/2}v(t)}^2-a\hnorm{\A^{1/2}v(0)}^2+\int_0^t \left\{\hnorm{v_t(\tau)}^2+a^2\hnorm{\A v(\tau)}^2\right\}d\tau.
\end{aligned}
\end{equation}
Applying \eqref{id:heat1}
to $v=\psi_{ttt}$ (i.e., $D_hv=w_{ttt}$),
to $v=\A^{1/2}\psi_{tt}$ (i.e., $D_hv=\A^{1/2} w_{tt}$),
and to $v=\A^{1/2}\psi_t$ (i.e., $D_hv=\A^{1/2} w_t$)
we obtain that the left hand side terms
under the time integrals in \eqref{est1} provide us with the estimates
\begin{align*}
\int_0^t \hnorm{w_{ttt}(\tau)}^2\, d\tau
=~&a\hnorm{\A^{1/2}\psi_{ttt}(t)}^2 -a\hnorm{\A^{1/2}\psi_{ttt}(0)}^2\\
&+\int_0^t \left\{\hnorm{\psi_{tttt}(\tau)}^2+a^2\hnorm{\A \psi_{ttt}(\tau)}^2\right\}d\tau,\\
\int_0^t \hnorm{\A^{1/2} w_{tt}(\tau)}^2\, d\tau
=~&a\hnorm{\A\psi_{tt}(t)}^2-a\hnorm{\A\psi_{tt}(0)}^2\\
&+\int_0^t \left\{\hnorm{\A^{1/2}\psi_{ttt}(\tau)}^2+a^2\hnorm{\A^{3/2}\psi_{tt}(\tau)}^2\right\}d\tau,\\
C_{\D(\A^{1/2})\hookrightarrow \h} \int_0^t \hnorm{\A w_t(\tau)}^2\, d\tau
\geq~& a\hnorm{\A\psi_{t}(t)}^2-a\hnorm{\A\psi_{t}(0)}^2\\
&+\int_0^t \left\{\hnorm{\A^{1/2}\psi_{tt}(\tau)}^2+a^2\hnorm{\A^{3/2}\psi_{t}(\tau)}^2\right\}d\tau.
\end{align*}
Inserting into \eqref{est1} we end up with
\begin{lemma} We have the estimate
\begin{equation}\label{est2}
\begin{aligned}
&\cE[w](t)+E[\psi](t)\\
&+\tilde{b}\int_0^t\Bigl\{
\hnorm{w_{ttt}(\tau)}^2+\hnorm{\A^{1/2}w_{tt}(\tau)}^2+\hnorm{\A w_t(\tau)}^2+\hnorm{\A w(\tau)}^2\\ &
\qquad+\hnorm{\psi_{tttt}(\tau)}^2+\hnorm{\A \psi_{ttt}(\tau)}^2+\hnorm{\A^{3/2}\psi_{tt}(\tau)}^2+\hnorm{\A^{3/2}\psi_t(\tau)}^2\Bigr\}d\tau\\
&\leq \tilde{C} \left(\cE[w](0)+E[\psi](0)+\int_0^t\left\{\hnorm{-\sigma(\psi_t^2)_{ttt}(\tau)}^2+\hnorm{-\sigma(\psi_t^2)_{tt}}^2\right\}d\tau\right)
\end{aligned}
\end{equation}
with $\tilde{b}$ sufficiently small and $\tilde{C}$ sufficiently large.
\end{lemma}

It remains to estimate the (quadratic, hence small for small $\psi$) terms on the right-hand side in terms of the left-hand side.
We have
for $\tau\in(0,t)$
\begin{equation*}
\begin{aligned}
\hnorm{-\sigma (\psi_t^2)_{ttt}(\tau))^2}
&\leq 4\sigma^2\hnorm{\psi_t \psi_{tttt}(\tau)}^2
+12\sigma^2\hnorm{\psi_{tt} \psi_{ttt}(\tau)}^2\\
&\leq 4\sigma^2 C_{\DA\hookrightarrow L^\infty}^2\sup_{s\in(0,t)} \hnorm{\A\psi_t(s)}^2 \hnorm{\psi_{tttt}(\tau)}^2\\
&\quad+12\sigma^2 C_{\DA\hookrightarrow L^\infty}^2\sup_{s\in(0,t)} \hnorm{\A\psi_{tt}(s)}^2 \hnorm{\psi_{ttt}(\tau)}^2
\end{aligned}
\end{equation*}
where we have used the embedding \eqref{linfDA}.

Similarly we get
\begin{equation*}
\begin{aligned}
\hnorm{-\sigma(\psi_t^2)_{tt}(\tau)}^2
&\leq 4\sigma^2\hnorm{\psi_t \psi_{ttt}(\tau)}^2
+4\sigma^2\hnorm{(\psi_{tt}(\tau))^2}^2\\
\nonumber&\leq 4\sigma^2 C_{\DA\hookrightarrow L^\infty}^2\sup_{s\in(0,t)} \hnorm{\A\psi_t(s)}^2 \hnorm{\psi_{ttt}(\tau)}^2\\
\nonumber&\quad+4\sigma^2 C_{\DA\hookrightarrow L^\infty}^2\sup_{s\in(0,t)} \hnorm{\A\psi_{tt}(s)}^2 \hnorm{\psi_{tt}(\tau)}^2.
\end{aligned}
\end{equation*}
Hence
\begin{align*}
&\int_0^t\left\{
 \hnorm{(\psi_t^2)_{ttt}(\tau)}^2+\hnorm{(\psi_t^2)_{tt}(\tau)}^2\right\}d\tau\\
&\leq 24\sigma^2 C_{\DA \hookrightarrow L^\infty}^2 \sup_{s\in(0,t)} E[\psi](s)
\int_0^t\left\{\hnorm{\psi_{tttt}(\tau)}^2+\hnorm{\psi_{ttt}(\tau)}^2+\hnorm{\psi_{tt}(\tau)}^2
\right\} d\tau
\end{align*}
holds. Inserting the latter into \eqref{est2} and using the fact that under the integral on the left-hand side we have
\begin{equation*}
\hnorm{\A^{1/2}w_{tt}(\tau)}^2+\hnorm{\A w_t(\tau)}^2+\hnorm{\A w(\tau)}^2
\geq \frac{2}{\max\{1,C_{\D(\A^{1/2})\hookrightarrow\h}^2\}} \cE[w](\tau)
\end{equation*}
and
\begin{equation*}
\hnorm{\A \psi_{ttt}(\tau)}^2
+\hnorm{\A^{3/2}\psi_{tt}(\tau)}^2+\hnorm{\A^{3/2}\psi_t(\tau)}^2
\geq \frac{2}{C_{\D(\A^{1/2})\hookrightarrow\h}^2} E[\psi](\tau)
\end{equation*}
yields the estimate
\begin{equation}\label{est3}
\begin{aligned}
&\cE[w](t)+E[\psi](t) +\bar{b}\int_0^t\left\{
r(\tau)+\cE[w](\tau)+E[\psi](\tau)+e[\psi](\tau)
\right\}d\tau\\
&\leq \bar{C} \left(\cE[w](0)+E[\psi](0)+\sup_{s\in(0,t)}E[\psi](s)\int_0^t e[\psi](\tau)d\tau\right)
\end{aligned}
\end{equation}
where we abbreviated
\begin{align}\label{energy:r}
r(t):=~&\hnorm{w_{ttt}(t)}^2+\hnorm{\A^{1/2}w_{tt}(t)}^2+\hnorm{\A w_t(t)}^2+\hnorm{\A w(t)}^2\\
&+\hnorm{\psi_{tttt}(t)}^2+\hnorm{\A \psi_{ttt}(t)}^2+\hnorm{\A^{3/2}\psi_{tt}(t)}^2+\hnorm{\A^{3/2}\psi_t(t)}^2,\nonumber\\
\label{energy:e}
e[\psi](t):=~&\hnorm{\psi_{tttt}(t)}^2+\hnorm{\A^{1/2}\psi_{ttt}(t)}^2+\hnorm{\A^{3/2}\psi_{tt}(t)}^2.
\end{align}
Using \eqref{energy:sum}, i.e. $\Lambda(t):=E[\psi](t)+\cE[w](t)$, yields the final result for this section.
\begin{proposition}\label{prop:energy}
The estimate
\begin{equation}\label{energyestimate:final}
\begin{aligned}
&\Lambda(t) +\bar{b}\int_0^t\left\{
r(\tau)+\Lambda(\tau)+e[\psi](\tau)
\right\}d\tau\\
&\leq \bar{C} \left(\Lambda(0)+\sup_{s\in(0,t)}E[\psi](s)\int_0^t e[\psi](\tau)d\tau\right)
\end{aligned}
\end{equation}
holds with $\bar{b}>0$ sufficiently small and $\bar{C}>0$ sufficiently large.
\end{proposition}

\section{Well-posedness of the nonlinear problem}
\label{sec:nonlinear}
\subsection{The nonlinear problem}
In this section we consider the initial boundary value problem \eqref{abbreq}
on an open and bounded domain $\Om \subseteq \R^d$, $d\in\{1,2,3\}$,
with smooth boundary $\partial\Om$. For now we restrict ourselves to a finite time interval,
i.e. $t\in[0,T]$ where $T<\infty$.

We again use the differential operators
\begin{equation*}
D_h := a\Delta - \partial_t
\quad\mbox{and}\quad
D_w := \partial_t^2-c^2\Delta-b\Delta\partial_t.
\end{equation*}
where this time we inserted $\A=-\Delta$. With this notation, the linearized version of \eqref{abbreq} reads
\begin{equation*}
D_h D_w \psi = f.
\end{equation*}
Explicitly, we have
\begin{align}
D_w\psi&=\psi_{tt}-c^2\Delta\psi - b\Delta\psi_t=\tilde{f},
\label{westervelteqn} \\ \label{heateqn}
D_h \tilde{f}&= a\Delta\tilde{f}-\tilde{f}_{t}= f.
\end{align}
Note that \eqref{heateqn} is the classical heat equation and \eqref{westervelteqn} is the
linearized Westervelt equation which has been studied, e.g., in \cite{KL09}
and can as well be used in order to describe elastic systems with structural damping (cf.\ \cite{ChenTriggiani:1989}).
Later on we will insert $f=\sigma(\psi_t^2)_{tt}$.
\begin{proposition}[Regularity results for the linearized Westervelt equation]
\label{reg:westervelt}
Consider the linearized Westervelt equation $D_w\psi=\ti{f}$. \\
(i) Suppose $\tilde{f} \in \liilii \cap H^1(0,T;H^{-1}(\Omega))$ and
			$\psi_0 \in \hii \cap H_0^1(\Om)$, $\psi_1 \in H_0^1(\Om)$, $\psi_2 \in \lii$. Then
			\begin{equation*}
			\psi \in C(0,T;\hii \cap\hio) \cap C^1(0,T;H_0^1(\Om)) \cap \ciilii \cap H^2(0,T;H_0^1(\Om)).
			\end{equation*}
(ii) If, in addition to (i), $\tilde{f} \in \clii$ and
			$\psi_1 \in \hii \cap\hio$, then we also have $\psi \in C^1(0,T;\hii \cap\hio)$.\\
(iii) If, in addition to (i), $\tilde{f} \in \hilii$ and
			$\psi_2 \in H_0^1(\Om)$, then we also have $\psi \in \hiiilii$.			
\end{proposition}
\begin{proof} (i) was proved in {\cite[Section 2.1]{KL09}}.\newline
(ii) It is clear that $b\Delta\psi_t=\psi_{tt}-c^2\Delta\psi-\tilde{f} \in \clii$ and thus
$\psi \in C^1(0,T;\hii \cap\hio)$.\newline
(iii) Lemma \ref{lem:Prop51} with $\h=L^2(\Om)$, $\A=-\Delta$, $f=\tilde{f}$ and $t=T$ gives us
\begin{equation*}
\frac{1}{2} \int_0^T \norm{\psi_{ttt}(\tau)}^2 \,d\tau + \frac{b}{2} \norm{\nabla \psi_{tt}(\tau)}^2 \Big|_0^T \leq
c^4 \int_0^T \norm{\Delta \psi_t(\tau)}^2\, d\tau + \int_0^T \norm{\tilde{f}_t(\tau)}^2 \, d\tau
\end{equation*}
and, invoking the assumptions and the fact that for any finite time horizon $T$ we have
$C^1(0,T;\hii \cap\hio) \subset H^1(0,T; \hii \cap\hio)$, we infer $\int_0^T\norm{\psi_{ttt}}^2< \infty$
and thus $\psi \in \hiiilii$ as claimed.
\end{proof}

\begin{proposition}[Regularity results for the heat equation]
\label{reg:heat}
Consider the heat equation $D_h \ti{f}=f$.
Suppose $f\in\liilii$ and $\ti{f}(0) \in \hio$. Then
\begin{equation*}
\ti{f} \in \chio \cap \hilii \cap L^2(0,T;\hii \cap \hio).
\end{equation*}
\end{proposition}
\begin{proof}
This improved regularity result for general second-order parabolic
equations has been proved, e.g., in \cite[Section 7.1, Theorem 5]{Evans:2010}.
\end{proof}
Note that, in order to combine Propositions \ref{reg:westervelt} and \ref{reg:heat},
we have to assume $\psi_0, \psi_1 \in H^3(\Om) \cap \hio$ which ensures
$\tilde{f}(0) \in \hio$ in Proposition \ref{reg:heat}.
\subsection{Local well-posedness}
Motivated by Propositions \ref{reg:westervelt} and \ref{reg:heat},
we define the space
\begin{equation}\label{detV}
\begin{aligned}
\mathcal{V}:=~& C^1(0,T;\hii \cap \hio) \cap \ciilii \\ & \cap H^2(0,T;H_0^1(\Om)) \cap \hiiilii
\end{aligned}
\end{equation}
equipped with the norm
\begin{equation*}
\|.\|_\mathcal{V}:= \|.\|_{\cihii} + \|.\|_{\ciilii}+\|.\|_{\hiihi} + \|.\|_{\hiiilii}.
\end{equation*}
We consider the equation
\begin{equation}\label{abstnonlineareqphi}
\left(a\Delta - \partial_t \right)
\left(\psi_{tt}-c^2\Delta\psi-b\Delta\psi_t \right) =
\sigma(\varphi_t)_{tt}^2.
\end{equation}
Our strategy in order to prove local existence of solutions of
\eqref{abbreq} is to apply Banach's Fixed Point Theorem to the map
\begin{equation*}
\mathcal{T}:\mathcal{W}\rightarrow\mathcal{V}
\end{equation*}
defined by
\begin{equation*}
\mathcal{T}(\varphi)=\psi
\end{equation*}
where $\psi$ is a solution of \eqref{abstnonlineareqphi} and the space $\mathcal{W}$
is given by
\begin{equation*}
\mathcal{W}=B_R^{\mathcal{V}}(0)=\{v\in \mathcal{V}: \|v\|_{\mathcal{V}} \leq \bar{m}\}
\end{equation*}
where $\bar{m}$ has to be suitably chosen later. Therefore, we need to show that $\mathcal{T}$ is a self-mapping,
$\mathcal{W}$ is closed and furthermore, that $\mathcal{T}$ is a contraction.

\vspace{5pt}
\noindent\textbf{Step 1.}  $\mathcal{T}: \mathcal{W} \rightarrow \mathcal{V}$ is a self-mapping.

This can be achieved by taking sufficiently
small initial data.
\begin{lemma}\label{lem:estf}
Let $\varphi \in \mathcal{V}$. Then $f= \sigma(\varphi_t)^2_{tt} \in \liilii$ and
\begin{equation*}
\|f\|_{\liilii} \leq C \|\varphi\|_{\mathcal{V}}^2.
\end{equation*}
\end{lemma}
\begin{proof}
We have $f= \sigma(\varphi_t)^2_{tt}=2\sigma(\varphi_{tt}^2+\varphi_t\varphi_{ttt})$.
As $\varphi\in \mathcal{V}$, we also have $\varphi \in \hiiilii$ and thus $\varphi_{ttt}\in \liilii$.
Furthermore, $\varphi \in \mathcal{V}$ implies $\varphi \in C^1(0,T;\hii \cap \hio)$.
Therefore $\varphi_{t} \in C(0,T;\hii \cap \hio)$.
Due to the fact that $\hii \cap \hio \subset L^\infty(\Omega)$ we arrive at $\varphi_t\in C(0,T;L^\infty(\Om))$.
Recalling that $C(0,T;L^\infty(\Om))$ is an ideal in $\liilii$ we conclude $\varphi_t \varphi_{ttt}
\in \liilii$. Therefore we estimate
\begin{align*}\|\varphi_t \varphi_{ttt}\|_{\liilii} & \leq  \|\varphi_t\|_{C(0,T;L^\infty(\Om))} \|\varphi_{ttt}\|_{\liilii}\\
\nonumber & \leq  \|\varphi\|_{C^1(0,T;L^\infty(\Om))}\|\varphi\|_{\hiiilii} \\
\nonumber & \leq  C_{H^2\hookrightarrow L^\infty} \|\varphi\|_{\cihii} \|\varphi\|_{\hiiilii}\\
\nonumber & \leq  C_{H^2\hookrightarrow L^\infty} \|\varphi\|_\mathcal{V}^2
\end{align*}
where $C_{H^2\hookrightarrow L^\infty}$ is the norm of the
embedding $H^2(\Om)\cap \hio\hookrightarrow L^\infty(\Om)$.

Next, we claim $\varphi_{tt}^2 \in \liilii$.
Our aim is to obtain $\varphi_{tt} \in L^4(0,T,L^4(\Om))$.
As $\varphi \in \mathcal{V}$, we have in particular $\varphi_{tt} \in L^2(0,T;H_0^1(\Om)) \cap \hilii$.
We use the Besov spaces
\begin{align*}
B^{1-\theta;2,2}(\Om)&=\left( L^2(\Om), W^{1,2}(\Om)\right)_{1-\theta,q;J},\\
B^{\theta;2,2}(0,T)&=\left( L^2(0,T), W^{1,2}(0,T)\right)_{\theta,q;J}.
\end{align*}
From Remark \ref{rem:Besov} we see that $B^{\theta;2,2}(\Om)=H^{\theta}(\Om)$ and
$B^{1-\theta;2,2}(0,T)=H^{1-\theta}(0,T)$. Therefore
$\varphi_{tt} \in H^{1-\theta}(0,T;H^\theta(\Om))$.
Due to the Sobolev Embedding Theorem, with the specific choice $\theta=\frac{3}{4}$,
we obtain $\varphi_{tt} \in H^{1/4}(0,T;H^{3/4}(\Om))\hookrightarrow L^4(0,T; L^4(\Om))$
and thus $\varphi_{tt}^2 \in \liilii$ as claimed. Therefore we may estimate
\begin{equation}\label{interpolate}
\begin{aligned}
\|\varphi_{tt}^2\|_{\liilii} & =\|\varphi_{tt}\|_{L^4(0,T;L^4(\Om))}^2\\
&\leq C_{H^{1/4}\hookrightarrow L^4}^2 C_{H^{3/4}\hookrightarrow L^4}^2
\|\varphi_{tt}\|_{H^{1/4}(0,T;H^{3/4}(\Om))}^2 \\
&\leq C_{H^{1/4}\hookrightarrow L^4}^2 C_{H^{3/4}\hookrightarrow L^4}^2
\|\varphi_{tt}\|_{\liihi}^{3/2} \|\varphi_{tt}\|_{\hilii}^{1/2} \\
&\leq C_{H^{1/4}\hookrightarrow L^4}^2 C_{H^{3/4}\hookrightarrow L^4}^2
\|\varphi\|_{\mathcal{V}}^2.
\end{aligned}
\end{equation}
Here, $C_{H^{1/4}\hookrightarrow L^4}^2$ and $C_{H^{3/4}\hookrightarrow L^4}^2$ denote
the norms of the embeddings $H^{1/4}(0,T)\hookrightarrow L^4(0,T)$ and
$H^{3/4}(\Om)\hookrightarrow L^4(\Om)$, respectively. Altogether, we have the desired estimate \eqref{lem:estf}
where $C=2\sigma\left(C_{H^{1/4}\hookrightarrow L^4}^2 C_{H^{3/4}\hookrightarrow L^4}^2+
C_{H^2\hookrightarrow L^\infty}\right)$.
\end{proof}

From Proposition \ref{reg:heat} and Lemma \ref{lem:estf}, we get
\begin{align}
\|\tilde{f}\|_{C(0,T;H^1(\Om)) \cap \hilii \cap \liihii} \nonumber
&\leq C_1 \big(\|f\|_{\liilii}+\|\tilde{f}(0)\|_{H^1(\Om)}\big)\\
&\leq C_1 \big(C\|\varphi\|_{\mathcal{V}}^2+\|\tilde{f}(0)\|_{H^1(\Om)}\big)\nonumber
\end{align}
for some constant $C_1$. Due to Proposition \ref{reg:westervelt}, we obtain
\begin{align}
\|\psi\|_{\mathcal{V}}
&\leq C_2 \big(\|\tilde{f}\|_{\hilii} + \|\psi_0\|_{H^2(\Om)}+\|\psi_1\|_{H^2(\Om)}+\|\psi_2\|_{\hi}\big)\nonumber\\
&\leq C_2 \big(C_1 \big(C\|\varphi\|_{\mathcal{V}}^2+\|\tilde{f}(0)\|_{H^1(\Om)}\big)+
\|\psi_0\|_{H^2(\Om)}+\|\psi_1\|_{H^2(\Om)}+\|\psi_2\|_{\hi}\big)\label{eqn:psi}
\end{align}
for some other constant $C_2$.
Suppose $\varphi\in\mathcal{W}$, so in particular $\|\varphi\|_{\mathcal{V}}\leq \bar{m}$.
By taking sufficiently small initial data,
\begin{equation}\label{bound:initialdata}
C_1\|\tilde{f}(0)\|_{H^1(\Om)}+\|\psi_0\|_{H^2(\Om)}+\|\psi_1\|_{H^2(\Om)}+\|\psi_2\|_{\hi}\leq r
\end{equation}
where $r=\frac{1}{4C C_1 C_2}$, we get $\|\psi\|_{\mathcal{V}}\leq \bar{m}$ provided
$\bar{m}\leq\frac{1}{2C C_1C_2}$. Thus, $\mathcal{T}\mathcal{W}\subseteq\mathcal{W}$.

\vspace{2pt}
\noindent\textbf{Step 2.} $\mathcal{W}$ is closed in $\mathcal{V}$. This step is clear as $\mathcal{W}$ is by its definition
a closed ball in $\mathcal{W}$.

\vspace{5pt}
\noindent\textbf{Step 3.} $\mathcal{T}:\mathcal{W}\rightarrow\mathcal{W}$ is a contraction.

It remains to show that $\mathcal{T}:\mathcal{W}\rightarrow\mathcal{W}$ is a
contraction, i.e. there exists a positive constant $\bar{C}<1$ such that
$\|\mathcal{T}(\varphi_1-\varphi_2)\|_\mathcal{V}\leq \bar{C} \|\varphi_1-\varphi_2\|_\mathcal{V}$
where $\varphi_i$, $\psi_i = \mathcal{T}\varphi_i$, $i=1,2$ solve
$D_hD_w\psi_i=\sigma(\varphi^2_{i_t})_{tt}$. Note that
$\hat{\varphi}=\varphi_1-\varphi_2$ and
$\hat{\psi}=\psi_1 - \psi_2 = \mathcal{T}\varphi_1 - \mathcal{T}\varphi_2$ solve
\begin{equation*}
(a\Delta - \partial_t )
(\hat{\psi}_{tt}-c^2\Delta\hat{\psi}-b\Delta\hat{\psi}_t) = \hat{f},
\end{equation*}
i.e. $D_w \hat{\psi}=\tilde{\hat{f}}$, $D_h\tilde{\hat{f}}=\hat{f}$
where $\hat{f}=\sigma\left(\hat{\varphi}_t\left(\varphi_{1_t}+\varphi_{2_t}\right)\right)_{tt}$.
\begin{lemma}\label{lem:hatf}
We have the estimate
\begin{equation*}
\|\hat{f}\|_{\liilii} \leq C\|\hat{\varphi}\|_\mathcal{V}
\left(\|\varphi_1\|_\mathcal{V}+\|\varphi_2\|_\mathcal{V}\right).
\end{equation*}
\end{lemma}
\begin{proof}
Explicitly, $\hat{f}=\sigma\left(\hat{\varphi}_{ttt}\left(\varphi_{1_t}+\varphi_{2_t}\right)+
2\hat{\varphi}_{tt}\left(\varphi_{1_{tt}}+\varphi_{2_{tt}}\right) +\hat{\varphi}_t\left(\varphi_{1_{ttt}}+\varphi_{2_{ttt}}\right)\right)$.
We are going to treat the three terms on the right-hand side separately and therefore estimate
{\allowdisplaybreaks
\begin{align}
&\|\hat{\varphi}_{ttt}\left(\varphi_{1_t}+\varphi_{2_t}\right)\|_{\liilii}\nonumber \\
& \leq  \|\hat{\varphi}_{ttt}\|_{\liilii)} \left(\|\varphi_{1_t}\|_{C(0,T;L^\infty(\Om)}+\|\varphi_{2_t}\|_{C(0,T;L^\infty(\Om)}\right)\nonumber\\
&\leq  C_{H^2\hookrightarrow L^\infty} \|\hat{\varphi}\|_{\hiiilii} \left(\|\varphi_{1}\|_{\cihii}+\|\varphi_{2}\|_{\cihii}\right)\nonumber\\
&\leq  C_{H^2\hookrightarrow L^\infty}\|\hat{\varphi}\|_\mathcal{V} \left(\|\varphi_{1}\|_\mathcal{V}+\|\varphi_{2}\|_\mathcal{V}\right)\nonumber\\[2mm]
&\|\hat{\varphi}_t\left(\varphi_{1_{ttt}}+\varphi_{2_{ttt}}\right)\|_{\liilii}\nonumber \\
&\leq C_{H^2\hookrightarrow L^\infty} \|\hat{\varphi}\|_{\cihii} \left(\|\varphi_1\|_{\hiiilii}+\|\varphi_1\|_{\hiiilii}\right)\nonumber \\
&\leq C_{H^2\hookrightarrow L^\infty} \|\hat{\varphi}\|_\mathcal{V}
\left(\|\varphi_1\|_\mathcal{V}+\|\varphi_1\|_\mathcal{V}\right), \nonumber \\[2mm] 
& \|\hat{\varphi}_{tt}\left(\varphi_{1_{tt}}+\varphi_{2_{tt}}\right)\|_{\liilii}\nonumber \\*
&\leq \|\hat{\varphi}_{tt}\|_{L^4(0,T;L^4(\Om))}\left(\|\varphi_{1_{tt}}\|_{L^4(0,T;L^4(\Om))}
+\|\varphi_{2_{tt}}\|_{L^4(0,T;L^4(\Om))}\right),\nonumber\\*
&\leq C_{H^{1/4}\hookrightarrow L^4}^2 C_{H^{3/4}\hookrightarrow L^4}^2\|\hat{\varphi}_{tt}\|_{H^{1/4}(0,T;H^{3/4}(\Om))} \nonumber\\*
& ~~~~~ \cdot \left(\|\varphi_{1}\|_{H^{1/4}(0,T;H^{3/4}(\Om))}+\|\varphi_{2}\|_{H^{1/4}(0,T;H^{3/4}(\Om))}\right)\nonumber\\*
&\leq C_{H^{1/4}\hookrightarrow L^4}^2 C_{H^{3/4}\hookrightarrow L^4}^2
\|\hat{\varphi}_{tt}\|_{\liihi}^{3/4} \|\hat{\varphi}_{tt}\|_{\hilii}^{1/4}\nonumber\\*
& ~~~~~ \cdot \left(\|\varphi_{1_{tt}}\|_{\liihi}^{3/4} \|\varphi_{1_{tt}}\|_{\hilii}^{1/4} +\|\varphi_{2_{tt}}\|_{\liihi}^{3/4} \|\varphi_{2_{tt}}\|_{\hilii}^{1/4}\right)\nonumber\\*
&\leq C_{H^{1/4}\hookrightarrow L^4}^2 C_{H^{3/4}\hookrightarrow L^4}^2
\|\hat{\varphi}\|_\mathcal{V}\left(\|\varphi_1\|_\mathcal{V}+\|\varphi_2\|_\mathcal{V}\right)\nonumber
\end{align}
}
where $C_{H^2\hookrightarrow L^\infty}$, $C_{H^{3/4}\hookrightarrow L^4}$ and
$C_{H^{1/4}\hookrightarrow L^4}$ denote the norms of the embeddings
$\hii \cap\hio \hookrightarrow L^\infty(\Om)$, $H^{3/4}(\Om)\hookrightarrow L^4(\Om)$ and $H^{1/4}(0,T)\hookrightarrow L^4(0,T)$,
respectively (cf. Lemma \ref{lem:estf}). Note that the third estimate is similar to \eqref{interpolate}.
Altogether, we have
\begin{equation*}
\|\hat{f}\|_{\liilii}
\leq C\|\hat{\varphi}\|_\mathcal{V}\left(\|\varphi_1\|_\mathcal{V}+\|\varphi_2\|_\mathcal{V}\right)
\end{equation*}
with $C$ the same constant as in the end of the proof of Lemma \ref{lem:estf}.
\end{proof}

Now, note that $\hat{\psi}(0)=0$, $\hat{\psi}_t(0)=0$, $\hat{\psi}_{tt}(0)=0$ and thus
$\tilde{\hat{f}}(0)=D_w \hat{\psi}(0)=0$. Therefore, by Proposition \ref{reg:heat} and Lemma \ref{lem:hatf}, we obtain
\begin{align*}
\|\tilde{\hat{f}}\|_{C(0,T;H^1(\Om)) \cap \hilii \cap \liihii} &\leq C_1 \|\hat{f}\|_{\liilii}\\
&\leq CC_1\|\hat{\varphi}\|_\mathcal{V}\left(\|\varphi_1\|_\mathcal{V}+\|\varphi_2\|_\mathcal{V}\right)\\
&\leq 2CC_1\bar{m}\|\hat{\varphi}\|_\mathcal{V}.\nonumber
\end{align*}
Invoking Proposition \ref{reg:westervelt} leads to $\|\hat{\psi}\|_\mathcal{V} \leq 2C C_1C_2\bar{m}\|\hat{\varphi}\|_\mathcal{V}$
and choosing $\bar{m}< \frac{1}{2CC_1C_2}$ we finally have contractivity. \vspace{3pt}
\begin{theorem} \label{thm:localexist}
Suppose $\psi_0, \psi_1 \in H^3(\Om)\cap H_0^1(\Omega)$ and $\psi_2 \in H_0^1(\Om)$.
For any $T>0$ there is a $\kappa_T>0$ such that if
\begin{equation}\label{bound:initialdata2}
\|\psi_0\|_{H^3(\Om)}+\|\psi_1\|_{H^3(\Om)}+\|\psi_2\|_{\hi}\leq \kappa_T,
\end{equation}
there exists a unique weak solution $\psi \in \mathcal{W}:=\{\psi \in \mathcal{V}: \norm{\psi}_{\mathcal{V}}\leq \bar{m}\}$ where $\bar{m}$
is sufficiently small and the space $\mathcal{V}$ is given by \eqref{detV}.
\end{theorem}
\begin{proof}
As the map $\mathcal{T}$ is a self-mapping and a contraction on $\mathcal{W}$ and moreover, $\mathcal{W}$ is a
closed subset of $\mathcal{V}$, we conclude that on the strength of the Banach Fixed Point Theorem
$\mathcal{T}: \mathcal{W}\rightarrow \mathcal{W}$ has a unique fixed point, i.e.
there exists a unique solution $\psi\in \mathcal{W}$ such that $\mathcal{T}(\psi)=\psi$.
Condition \eqref{bound:initialdata2} comes from \eqref{bound:initialdata}.
\end{proof}

\subsection{Global well-posedness} Our strategy in order to prove global in-time well-posedness of equation
\eqref{abbreq} is to use the classical barrier method.

Let $\psi$ be a local in time solution of \eqref{abbreq} according to Theorem \ref{thm:localexist} and let $t<T$,
where $T$ is the maximal existence time (possibly $T=\infty$).
The barrier method now yields global in time existence for sufficiently small initial data.
\begin{theorem}
For all initial values $\psi_0$, $\psi_1$ and $\psi_2$ of $\psi$, $\psi_t$ and $\psi_{tt}$ satisfying
\begin{equation}\label{Lambdarho}
\Lambda(0)\leq \rho
\end{equation}
with $\rho$ sufficiently small,
\begin{equation}\label{rho}
\rho\leq \frac{\bar{b}}{2 \bar{C} \max\{1,\bar{C}\}},
\end{equation}
we get that for all $t>0$
\begin{equation}\label{est:Lambda}
\Lambda(t)\leq 2\max\{1,\bar{C}\}\rho.
\end{equation}
In particular, there exists a (sufficiently small) constant $\bar{\rho} >0$ such that, if for the initial values
$\psi_0$, $\psi_1$ and $\psi_2$ we have
\begin{align*}
&\norm{\psi_1}_\hii+\norm{\psi_2}_\hii+\norm{\psi_{ttt}(0)}_\hi\\
&+\norm{w(0)}_\hii+\norm{w_t(0)}_\hi +\norm{w_{tt}(0)}_\hi \leq \bar{\rho},
\end{align*}
where $w=\psi_{t}(t)-a\Delta\psi(t)$, then for all times $T>0$ there exists a unique solution
\begin{align*}
\psi \in~& C^2(0,T; \hii\cap \hio) \cap C^3(0,T;H_0^1(\Om)) \cap H^2(0,T;\hiii \cap \hio)\\
& \cap H^3(0,T;\hii\cap \hio) \cap H^4(0,T;\lii)
\end{align*}
such that furthermore
\begin{align*}
w \in~& C(0,T;\hii\cap \hio) \cap C^2(0,T;H_0^1(\Om)) \cap H^1(0,T;\hii\cap \hio)\\
& \cap H^2(0,T;H_0^1(\Om)) \cap H^3(0,T;\lii).
\end{align*}
\end{theorem}
\begin{proof}
Assume that there exists a finite time such that \eqref{est:Lambda} is violated. We denote by $T$ the
minimal such time and observe that $T>0$ since $\Lambda(0)< 2\max\{1,\bar{C}\}\rho$. Then we have
\begin{equation}\label{ass:fail}
\Lambda(T)\geq 2\max\{1,\bar{C}\}\rho
\end{equation}
and moreover, \eqref{est:Lambda} holds for all $t\in(0,T)$. From \eqref{est3} we get the estimate
\begin{equation}\label{est:global}
\begin{aligned}
&\Lambda(t)+\bar{b}\int_0^t\left\{r(\tau)+\Lambda(\tau)+e[\psi](\tau)\right\}d\tau\\
&\leq \bar{C} \left(\Lambda(0)+2\max\{1,\bar{C}\}\rho\int_0^t e[\psi](\tau)d\tau\right)
\end{aligned}
\end{equation}
for all $t\in(0,T)$ which by $\Lambda(0)\leq\rho$ and $2\bar{C}\max\{1,\bar{C}\}\rho\leq\bar{b}$ gives
$\Lambda(t)\leq \bar{C}\rho$ for all $t\in(0,T)$, hence by continuity $\Lambda(T)\leq \bar{C}\rho$,
a contradiction to \eqref{ass:fail}.
This proves that the bound \eqref{est:Lambda} holds for all $t>0$ provided \eqref{Lambdarho}
holds with \eqref{rho}.

The regularity of the solution is obtained by \eqref{est:global} together with the definitions of
$\Lambda(t)$ and $r(t)$ in \eqref{energy:sum} and \eqref{energy:r}, respectively.
\end{proof}
\begin{remark}
The term $r(\tau)$ is not required for the proof of global well-posedness, but it provides us with higher order regularity.
\end{remark}

\begin{theorem} Provided  \eqref{est:Lambda} holds with \eqref{rho} we have
\begin{equation*}
\Lambda(t)\leq \E^{-\frac{\bar{b}}{\bar{C}} t} \Lambda(0)
\end{equation*}
where $\bar{b}$ and $\bar{C}$ are the same constants as in Proposition \ref{prop:energy}.
\end{theorem}
\begin{proof}
Inserting \eqref{est:Lambda} with \eqref{rho} into \eqref{est3} yields
\begin{equation*}
\Lambda(t)+\bar{b}\int_0^t\left\{r(\tau)+\Lambda(\tau)\right\}d\tau \leq \bar{C} \Lambda(0).
\end{equation*}
In particular, as $\Lambda(t)\geq 0$ and $r(\tau)\geq 0$, we obtain
\begin{equation*}
\int_0^t\Lambda(\tau)d\tau \leq \frac{\bar{C}}{\bar{b}}\Lambda(0).
\end{equation*}
The desired estimate follows now from equation (3) in the proof of Theorem 8.1 in \cite{Komornik:1994}
\end{proof}

\begin{appendix}
\section{Semigroup theory} We summarize some facts from the theory of operator semigroups.
For additional information on this topic we refer to \cite{EngelNagel:1999} and \cite{Pazy:1983}.
Suppose $(X,\|.\|)$ is a Banach space and denote $\Sigma_\delta := \{\lam \in \C: |\mathrm{arg}(\lam) |< \delta\}$

\begin{definition}[{\cite[Definition 1.1.1, 1.2.1]{Pazy:1983}}]
(i) A one parameter family $T(t)$, $t\in[0,\infty)$, of bounded linear operators from $X$ into $X$
is called a semigroup of bounded linear operators on $X$ if $T(0)=I$ and $T(t+s)=T(t)T(s)$ for every $t,s \geq 0$.

The linear operator $A$ defined by
\begin{equation*}
Ax=\lim_{t\downarrow 0} \frac{T(t)x-x}{t}, \qquad \D(A)=\left\{x\in X: \lim_{t\downarrow 0} \frac{T(t)x-x}{t}\mbox{ exists}\right\}
\end{equation*}
is called the infinitesimal generator of the semigroup $T(t)$, with domain $\D(A)$.\\
(ii) If, in addition to (i), $\lim_{t\downarrow 0} T(t)x=x$ for every $x\in X$, then $T(t)$ is called a strongly continuous semigroup of
bounded linear operators. \\
(iii) If, in addition to (i) and (ii), $\norm{T(t)} \leq 1$, then $T(t)$ is called a strongly continuous semigroup of contractions.
\end{definition}

\begin{definition}[{\cite[Definition II.3.13]{EngelNagel:1999}}]
A linear operator $(A,\D(A))$ on a Banach space $X$ is called dissipative if
$\norm{(\lam I-A)x} \geq \lam \norm{x}$ for all $\lam >0$ and $x\in \D(A)$.
\end{definition}

\begin{theorem}[Lumer-Phillips, {\cite[Theorem 1.4.3]{Pazy:1983}}]
Let $(A,\D(A))$ be a densely defined linear operator on a Banach space $X$.\\
(i) If $A$ is dissipative and there is a $\lam_0 >0$ such that $\mathrm{Range}(\lam_0 I - A)=X$, then
$A$ is the infinitesimal generator of a strongly continuous semigroup of contractions.\\
(ii) If $A$ is the infinitesimal generator of a strongly continuous semigroup of contractions on $X$,
then $\mathrm{Range}(\lam I -A)=X$ for all $\lam >0$ and $A$ is dissipative.
\end{theorem}

\begin{definition}[{\cite[Definition II.4.5]{EngelNagel:1999}}]
A family $(T(z))_{z\in \Sigma_\delta \cup\{0\}}$ of bounded linear operators on $X$ is called an
analytic semigroup (of angle $\delta \in (0,\pi/2]$) if \\
(i) $T(0)=I$ and $T(z_1+z_2)=T(z_1)T(z_2)$ for all $z_1,z_2 \in \Sigma_\delta$,\\
(ii) the map $z\mapsto T(z)$ is analytic in $\Sigma_\delta$, \\
(iii) $\lim_{\Sigma_{\delta'} \ni z \rightarrow 0}T(z)x=x$ for all $x\in X$ and $0<\delta'<\delta$.\\
If, in addition, $\|T(z)\|$ is bounded in $\Sigma_{\delta'}$ for every $0< \delta'<\delta$, we call
$(T(z))_{z\in \Sigma_\delta \cup\{0\}}$ a bounded analytic semigroup.
\end{definition}

\begin{theorem}[{\cite[Theorem II.4.6]{EngelNagel:1999}}]\label{thm:equivanalytic}
For an operator $(A, \D(A))$ on a Banach space X, the following statements are equivalent:\\
(i) $A$ generates a bounded analytic semigroup $(T(z))_{z\in \Sigma_\delta \cup\{0\}}$ on X.\\
(ii) $A$ generates a bounded strongly continuous semigroup $(T(t))_{t \geq 0}$ on $X$, and there
exists a constant $C>0$ such that the resolvent of $A$ satisfies
\begin{equation*}
\norm{R(r+\I s,A)}\leq \frac{C}{|s|}
\end{equation*}
for all $r>0$ and $0\neq s \in \R$.\\
(iii) $A$ is sectorial, i.e. there exists $0<\delta\leq \frac{\pi}{2}$ such that $\Sigma_{\frac{\pi}{2}+\delta} \setminus \{0\} \subset \rho(A)$ and for
each $\varepsilon \in (0,\delta)$ there exists $M_\varepsilon >1$ such that $\|R(\lam, A)\| \leq M_\varepsilon/|\lam|$ for all $0 \neq \lam \in
\overline{\Sigma}_{\frac{\pi}{2}+\delta - \varepsilon}$.
\end{theorem}

\begin{lemma}[{\cite[Corollary II.4.7]{EngelNagel:1999}}]\label{lem:normalop}
If $A$ is a normal operator on a Hilbert space $H$ satisfying $\sig(A) \subset \{z\in\C: \mathrm{arg}(-z) < \delta\}$
for some $\delta \in [0, \pi/2)$, then $A$ generates a bounded analytic semigroup.
\end{lemma}

\begin{definition}[{\cite[Definition III.2.1]{EngelNagel:1999}}]
Let $A:\D(A) \rightarrow X$ be a linear operator. An operator $B: \D(B) \rightarrow X$ is
called (relatively) $A$-bounded if $\DA \subseteq \D(B)$ and if there exist positive constants $\alpha, \beta \geq 0$ such that
\begin{equation}\label{relbound}
\|Bx\| \leq \alpha \|Ax\|+\beta \|Bx\|
\end{equation}
for all $x\in\D(A)$. The $A$-bound of $B$ is
\begin{equation*}
a_0 := \inf\{\alpha \geq 0: \mbox{there exists }\beta > 0\mbox{ such that }\eqref{relbound}\mbox{ holds}\}.
\end{equation*}
\end{definition}

\begin{proposition}[{\cite[Theorem III.2.10]{EngelNagel:1999}}] \label{prop:perturb}
Suppose the operator $\left(A, \D(A)\right)$ generates an analytic semigroup $(T(z))_{z \in \Sigma_\delta \cup \{0\}}$ on
$X$. Then there exists a constant $\alpha >0$ such that $\left(A+B,\D(A)\right)$ generates an analytic
semigroup for every $A$-bounded operator $B$ having $A$-bound $\alpha_0 < \alpha$.
\end{proposition}

\begin{theorem}[{\cite[Theorem 4.4.3]{Pazy:1983}}]\label{thm:analyticdecay}
An analytic semigroup $(T(z))_{z\in \Sigma_\delta \cup\{0\}}$ is uniformly exponentially stable, i.e. there exist constants
$M\geq 1$ and $\om >0$ such that $\|T(t)\|\leq M \E^{-\om t}$ if and only if the spectral bound $s(A):=\sup\{\re(\lam):\lam\in\sig(A)\}$
of its generator $A$ satisfies $s(A)<0$.
\end{theorem}
\end{appendix}

\bigskip
\noindent
{\bf Acknowledgments.}
We thank Pedro Jordan for drawing our attention to the model equation investigated in the present paper. 
Furthermore, we gratefully acknowledge the financial support of our research by the FWF (Austrian Science Fund) under grant P24970.



\begin{thebibliography}{15}
\bibitem{AdamsFournier:2003}
\newblock R.\ A.\ Adams and J.\ F.\ Fournier, 
\newblock \emph{Sobolev Spaces, Second Edition},
\newblock Elsevier/Academic Press, Amsterdam, 2003.

\bibitem{ChenRussell:1982}
\newblock G.\ Chen and D.\ L.\ Russell, 
\newblock {A mathematical model for linear elastic systems with structural damping}, 
\newblock \emph{Quarterly of Applied Mathematics}, \textbf{39} (1981), 433--454.

\bibitem{ChenTriggiani:1989} 
\newblock S.\ Chen and R.\ Triggiani, 
\newblock {Proof of extensions of two conjectures on structural damping for elastic systems}, 
\newblock \emph{Pacific Journal of Mathematics}, \textbf{136} (1989), 15--55.

\bibitem{Coulouvrat:1992}
\newblock F.\ Coulouvrat, 
\newblock {On the equations of nonlinear acoustics}, 
\newblock \emph{Journal d'Acoustique}, \textbf{5} (1992), 321--359.

\bibitem{Crighton:1979}
\newblock D.\ G.\ Crighton, 
\newblock {Model equations of nonlinear acoustics}, 
\newblock \emph{Annual Review of Fluid Mechanics}, \textbf{11} (1979), 11--33.

\bibitem{EngelNagel:1999}
\newblock K.-J.\ Engel and R.\ Nagel, 
\newblock \emph{One-Parameter Semigroups for Linear Evolution Equations},
\newblock Springer, New York, 1999.

\bibitem{Evans:2010} 
\newblock L.\ C.\ Evans, 
\newblock \emph{Partial Differential Equations}, Second Edition, 
\newblock American Mathematical Society, Providence, 2010.

\bibitem{Fattorini:1983}
\newblock H.\ O.\ Fattorini, 
\newblock \emph{The Cauchy Problem}, 
\newblock Addison-Wesley, Massachusetts, 1983.

\bibitem{HamiltonBlackstock:1997}
\newblock M.\ F.\ Hamilton and D.\ T.\ Blackstock, 
\newblock \emph{Nonlinear Acoustics}, 
\newblock Academic Press, New York, 1997.

\bibitem{Jordan:2004}
\newblock P.\ M.\ Jordan, 
\newblock {An analytical study of Kuznetsov's equation: diffusive solutions, shock formation and solution bifurcation}, 
\newblock \emph{Physics Letters A}, \textbf{326} (2004), 77--85.

\bibitem{KL09} 
\newblock B.\ Kaltenbacher and I.\ Lasiecka, 
\newblock Global existence and exponential decay rates for the Westervelt equation, 
\newblock \emph{Discrete and Continuous Dynamical Systems Series S}, \textbf{2} (2009), 503--525.

\bibitem{KL10}
\newblock B.\ Kaltenbacher and I.\ Lasiecka, 
\newblock {An analysis of nonhomogeneous Kuznetsov's equation: Local and global well-posedness; exponential decay}, 
\newblock \emph{Mathematische Nachrichten}, \textbf{285} (2012), 295--321.

\bibitem{KL11}
\newblock B.\ Kaltenbacher and I.\ Lasiecka, 
\newblock Well-posedness of the Westervelt and the Kuznetsov equation with nonhomogeneous Neumann boundary conditions, 
\newblock \emph{DCDS Supplement, Proceedings of the 8th AIMS Conference}, (2011), 763--773.

\bibitem{KLM11}
\newblock B.\ Kaltenbacher, I.\ Lasiecka, and R.\ Marchand,
\newblock {Well-posedness and exponential decay rates for the Moore-Gibson-Thompson equation}, 
\newblock \emph{Control and Cybernetics}, \textbf{40} (2012), 971--988.

\bibitem{KLP12} 
\newblock B.\ Kaltenbacher, I.\ Lasiecka and M.\ K.\ Pospieszahlska, 
\newblock Well-posedness and exponential decay of the energy in the nonlinear Jordan-Moore-Gibson-Thompson equation arising in high intensity ultrasound, 
\newblock \emph{Mathematical Models and Methods in Applied Sciences}, \textbf{22} (2012), 1250035, 34 pages.

\bibitem{Kaltenbacher:2004}
\newblock M.\ Kaltenbacher, 
\newblock \emph{Numerical Simulation of Mechatronic Sensors and Actuators},
\newblock Springer, Berlin, 2004.

\bibitem{Komornik:1994}
\newblock V.\ Komornik, 
\newblock \emph{Exact Controllability and Stabilization. The Multiplier Method}, 
\newblock Masson-John Wiley, Paris-Chicester, 1994.

\bibitem{Kuznetsov}
\newblock V.\ P.\ Kuznetsov, 
\newblock {Equations of nonlinear acoustics}, 
\newblock \emph{Soviet physics. Acoustics}, \textbf{16} (1971), 467--470.

\bibitem{LiangXiao:1997}
\newblock J.\ Liang and T.\ Xiao, 
\newblock {Semigroups arising from elastic systems with dissipation}, 
\newblock \emph{Computers and Mathematics with Applications}, \textbf{33} (1997), 1--9.

\bibitem{DiNezzaPalatucciValdinoci:2011}
\newblock E.\ Di Nezza, G.\ Palatucci, and E.\ Valdinoci, 
\newblock {Hitchhiker's guide to the fractional Sobolev spaces}, 
\newblock \emph{Bulletin des Sciences Math\'ematiques}, \textbf{136} (2012), 521--573.

\bibitem{Pazy:1983} 
\newblock A.\ Pazy, 
\newblock \emph{Semigroups of Linear Operators and Applications to Partial Differential Equations}, 
\newblock Springer, New York, 1983.

\bibitem{Rozanova:2007} 
\newblock A.\ Rozanova, 
\newblock The Khokhlov-Zabolotskaya-Kuznetsov equation, 
\newblock \emph{Comptes Rendus Math{\'e}matique}, \textbf{344} (2007), 337--342.

\bibitem{Tjotta:2001}
\newblock S.\ Tj{\o}tta, 
\newblock {Higher order model equations in nonlinear acoustics}, 
\newblock \emph{Acta Acustica united with Acustica}, \textbf{87} (2001), 316--321.

\bibitem{Westervelt:1963}
\newblock P.\ J.\ Westervelt, 
\newblock {Parametric acoustic array}, 
\newblock \emph{Journal of the Acoustical Society of America}, \textbf{35} (1963), 535--537.

\end{thebibliography}
\end{document}